\documentclass{fundam}

\publyear{25}
\papernumber{7}
\volume{195}
\issue{1-4}
\theDOI{10.46298/fi.12449}

\versionForARXIV

\usepackage{amssymb, amsmath}
\usepackage{array}
\usepackage{url}
\usepackage{tikz}
\usepackage[kerning]{microtype}
\tikzstyle{every picture} = [scale=.45]
\tikzstyle{every node} = [draw, fill=white, circle, inner sep=0pt, minimum size=4pt]
\tikzstyle{d} = [very thick]
\tikzstyle{i} = [draw, fill=black, circle, inner sep=0pt, minimum size=4pt]
\tikzstyle{n} = [draw=none, rectangle, inner sep=2pt]

\let\:\colon
\let\cl\mathcal
\let\ct\mathsf
\let\le\leqslant
\let\m\mathbf
\let\nle\nleqslant
\let\subset\subseteq
\let\parmap\nrightarrow
\let\setminus\smallsetminus
\newcommand\category[1]{\expandafter\newcommand\csname#1\endcsname{\ct{#1}}}

\newcommand\nat{\xrightarrow{.}}
\newcommand\nein{\mathord{\sim}}
\newcommand\nbd[1]{\protect\nobreakdash#1\hspace{0pt}}
\newcommand\no{\mathord{-}}
\newcommand\op{\mathrm{op}}
\renewcommand\P{\mathcal{P}}
\newcommand\pair[1]{\langle#1\rangle}
\newcommand\SD{\mathrm{SD}}

\category{CAGBA}
\category{IIPOMn}
\category{LIIPOSg}
\category{Par}
\category{SDCAGBA}
\category{SDIIPOMn}
\category{SDPar}
\DeclareMathOperator\At{At}

\begin{document}

\title{Locally Integral Involutive PO-Semigroups}
\author{Jos\'e Gil-F\'erez\\
Chapman University, Orange CA 92866, USA\\
gilferez{@}chapman.edu
\and Peter Jipsen\\
Chapman University, Orange CA 92866, USA\\
jipsen{@}chapman.edu
\and Melissa Sugimoto\\
Chapman University, Orange CA 92866, USA\\
msugimoto{@}chapman.edu}
\date{\today}

\maketitle

\runninghead{J. Gil-F\'erez, P. Jipsen, M. Sugimoto}{Locally Integral IPO-Semigroups}

\vspace{-7ex}

\begin{abstract}
We show that every locally integral involutive partially ordered semigroup (ipo\nbd-semigroup) $\m A = (A,\le, \cdot, \nein,\no)$, and in particular every locally integral involutive semiring, decomposes in a unique way into a family $\{\m A_p : p\in A^+\}$ of integral ipo\nbd-monoids, which we call its \emph{integral components}. In the semiring case, the integral components are unital semirings. Moreover, we show that there is a family of monoid homomorphisms $\Phi = \{\varphi_{pq}\: \m A_p\to \m A_q : p\le q\}$, indexed over the positive cone $(A^+,\le)$, such that the structure of $\m A$ can be recovered as a glueing $\int_\Phi \m A_p$ of its integral components along $\Phi$. Reciprocally, we give necessary and sufficient conditions such that the P\l{}onka sum of any family of integral ipo\nbd-monoids $\{\m A_p : p\in D\}$, indexed over a join-semilattice $(D,\lor)$ along a family of monoid homomorphisms $\Phi$ is an ipo\nbd-semigroup.
\end{abstract}


\section{Introduction}

\emph{Idempotent semirings} are algebras of the form $(A,\lor,\cdot)$ where $(A,\lor)$ is a semilattice (with order $x\le y \iff x\lor y=y$), $(A,\cdot)$ is a semigroup, and the semigroup operation distributes over the join. They play an important role in mathematics, logic, and theoretical computer science, since they generalize distributive lattices and, if they have an identity element, they expand to Kleene algebras and residuated lattices. An \emph{involutive semiring} is an idempotent semiring with two antitone operations $\nein$ and $\no$ satisfying $\nein\no x=x=\no\nein x$ and
\[
x\cdot y\le z\iff y\cdot\nein z\le \nein x \iff \no z\cdot x\le \no y.
\]
These algebras are term-equivalent to identity-free involutive residuated lattices and, in the case that the lattice is complete, to Frobenius quantales (see~\cite{GaJiKoHi07} and~\cite{EGGHK18}). Furthermore, algebras of binary relations are involutive semirings under the operations of union, composition, and complement-converse. The structural characterization obtained in this paper is valid for more general partially ordered structures called \emph{involutive po-semigroups} (\emph{ipo\nbd-semigroups}) where the semilattice $(A,\lor)$ is replaced by a poset $(A,\le)$. In these structures, the product is residuated and its residuals $\backslash$ and $/$ are term-definable. Prominent examples of ipo\nbd-semigroups are (1-free reducts of) groups (where the poset is an antichain) and partially ordered groups (in both cases $\nein x=-x=x^{-1}$).

An ipo\nbd-semigroup has a \emph{global identity} if there is an element $1$ satisfying $x\cdot 1 = x = 1\cdot x$, which implies the inequality $y\le (x/x)y$. An ipo\nbd-semigroup is \emph{integral} when it possesses a global identity which is moreover the top element, that is, $x\le 1$. Integrality is an important property for residuated lattices, since it is equivalent to the proof-theoretical rule called weakening. Every integral ipo\nbd-semigroup is \emph{integrally closed} (see~\cite{GFLaMe20}), namely, $x\backslash x = 1 = x/x$, and satisfies $x\backslash 1 = 1$. In this work we investigate a much larger class of structures, which we call locally integral ipo\nbd-semigroups. A \emph{balanced} ipo\nbd-semigroup, (i.e., one satisfying $1_x:= x\backslash x = x/x$) is \emph{locally integral} if it satisfies $y\le 1_x y$ (and in particular, $1_x\cdot x = x$), $x\le 1_x$, and $x\backslash 1_x = 1_x$. For this reason, we call the elements of the form $1_x$ \emph{local identities}. An integral ipo\nbd-semigroup is a locally integral one satisfying $1_x = 1_y$.

The main result in this paper is that every locally integral ipo\nbd-semigroup $\m A$ can be decomposed in a unique way into a family of integral ipo\nbd-monoids $\{\m A_p : p\in A^+\}$, indexed on the set $A^+ = \{{p\in A} : \text{for every }x,\ x\le px\}$ of \emph{positive}\footnote{~In an ipo\nbd-semigroup with a global identity $1$, the positive elements coincide with the elements greater than 1.} elements of $\m A$, which we call its \emph{integral components}. Two locally integral ipo\nbd-semigroups can have the same integral components, but may differ in the way these components are \emph{glued} together. We find in the literature similar situations in which a number of structures are glued together to form a new one: for instance, in~\cite{GFJiMe20} it is described how chains can be attached to an odd Sugihara monoid in order to form a commutative idempotent residuated chain, and in~\cite{JiTuVa21} how Boolean algebras can be glued together to form commutative idempotent involutive residuated lattices.

In our present case, we associate to every locally integral ipo\nbd-semigroup $\m A$ the join-semilattice $\m A^+ = (A^+,\cdot)$ and a family of monoid homomorphisms $\Phi = \{\varphi_{pq}\: {\m A_p\to \m A_q} : p\le q \text{ in } A^+\}$ between its integral components such that the structure of $\m A$ can be completely recovered as an aggregate or \emph{glueing} $\int_\Phi \m A_p$ of these integral components along~$\Phi$ in two stages: first, the semigroup part of $\m A$ turns out to be the P\l{}onka sum of the family~$\Phi$, and the involutive negations can be defined componentwise. Then, we recover the order of $\m A$ using the product, the negations, and the local identities.

As an application of our results, we can combine certain semantics for fuzzy logics with semantics for relevance logic using, for example, the well-understood MV-algebras as building blocks of a glueing. We also exploit this decomposition in order to prove that several properties of locally integral ipo\nbd-semigroups are \emph{local}, in that a locally integral ipo\nbd-semigroup satisfies them if and only if all its integral components satisfy them. One of the most significant local properties established here is local finiteness.

Previous research into the structure of doubly-idempotent semirings can be found in~\cite{AlJi20,AJS21}. The structure of all finite commutative idempotent involutive residuated lattices is completely described in~\cite{JiTuVa21} in a step-by-step decomposition. In the current paper, this is significantly generalized to all locally integral ipo\nbd-semigroups, without any restrictions regarding finiteness, commutativity, or full idempotence. Similar results were originally proved for the special case of locally integral ipo\nbd-monoids \cite{GFJiLo23}. In the present form they are valid for locally integral Frobenius quantales (without a unit). We also characterize the ipo\nbd-semigroups that can be embedded into ipo\nbd-monoids. A further use of P\l{}onka sums can be found in~\cite{Je22}, where the structure of even and odd involutive commutative residuated chains is studied.

We set the terminology and notation in Section~2, and describe the fundamental properties of ipo\nbd-semigroups needed in the rest of the paper. In Section~3, we introduce the class of locally integral ipo\nbd-semigroups and show that every locally integral ipo\nbd-semigroup is the glueing of its integral components. Finally, in Section~4, we solve the reverse problem, that is, we provide necessary and sufficient conditions so that the glueing of a system of integral ipo\nbd-monoids is an ipo\nbd-semigroup. In the final section we consider the category of idempotent ipo\nbd-semigroups and describe its duality with a subcategory of semilattice directed systems of sets and partial maps.

\section{Involutive partially ordered semigroups and semirings}
\label{sec:ipo-semigroups}

An \emph{involutive partially ordered semigroup}, or \emph{ipo\nbd-semigroup} for short, is a structure of the form $\m A = (A,\le,\cdot,\nein,\no)$ such that $(A,\le)$ is a partially ordered set, $(A,\cdot)$ is a semigroup with two \emph{antitone} (i.e., order-reversing) unary operations $\nein$ and $\no$ satisfying
\emph{double negation} (dn)  $\nein\no x = x = \no\nein x$ and \emph{rotation}:
\[
x\cdot y \le z \iff y\cdot\nein z \le \nein x \iff \no z\cdot x\le \no y. \tag{rot}
\]
This is then a po-variety, in the sense of~\cite{Pi04}, namely, a class of partially ordered algebras that satisfy a set of equations and inequations. The unary operations $\nein$ and $\no$, which are allowed to coincide, are called \emph{involutive negations}. Indeed, we call $\m A$ \emph{cyclic} if it satisfies $\nein x = \no x$. An element $x$ of $\m A$ is \emph{central} if $x\cdot y = y \cdot x$ for any other $y\in A$, and $\m A$ is \emph{commutative} if all its elements are central. An element $x$ of $\m A$ is \emph{idempotent} if $x\cdot x = x$, and $\m A$ is \emph{idempotent} if all its elements are idempotent. From now on, if no confusion is likely, we will write $xy$ instead of $x\cdot y$.

We can readily check that, in the presence of antitonicity and double negation, rotation is equivalent to the following property of \emph{residuation}:
\[
xy\le z \iff x \le \no(y\cdot\nein z) \iff y \le \nein(\no z\cdot x). \tag{res}
\]
Thus, the multiplication of every ipo\nbd-semigroup is residuated in both arguments, with \emph{left} and \emph{right residuals} given by $z/y = \no(y\cdot\nein z)$ and $x\backslash z = \nein (\no z\cdot x)$, respectively. With this notation, (res) can be rewritten as follows:
\[
xy \le z \iff x \le z/y \iff y \le x\backslash z.
\]
An immediate consequence of this observation is that multiplication preserves arbitrary existing joins and is therefore order-preserving in both arguments. Also, an ipo\nbd-semigroup is commutative if and only if it satisfies $x\backslash y = y/x$.

\begin{lemma}\label{lem:other:properties:ipo-mon}\label{lem:part:swap:negations}
Every ipo\nbd-semigroup satisfies the following \emph{laws of contraposition}:
\[
y/x = \no y\backslash \no x\quad \text{and}\quad x\backslash y = \nein x/\nein y.
\]
\end{lemma}

\begin{proof}    
For any arbitrary $z$ we have the equivalences
\[
z \le y/x \iff zx \le y \iff \no y z \le \no x \iff z \le \no y\backslash \no x,
\]
by residuation and rotation, and therefore the first equation follows. The second equation is a consequence of the first one and double negation.
\end{proof}

Using the involutive negations and the multiplication, we can define a binary operation $+$ on $A$ in two different ways: $x+y = \nein(\no x\cdot\no y)$ and $x+y = \no(\nein x\cdot \nein y)$. The previous lemma guarantees that both definitions agree, since
\[
\nein(\no x\cdot \no y) = \no y\backslash x = y/\nein x = \no(\nein x\cdot\nein y).
\]

An element $p$ of an ipo\nbd-semigroup $\m A$ is \emph{positive} if for all $x\in A$, $x\le px$, or equivalently, for all $x\in A$, $x\le xp$. Indeed, suppose that for all $x\in A$, the element $p$ satisfies $x\le px$. In particular, $\nein(xp) \le p\cdot\nein(xp)$, whence we obtain that
\[
x\le (xp)/p = \no(p\cdot\nein(xp)) \le \no\nein(xp) = xp,
\]
by residuation, antitonicity, and double negation. The other implication is analogous. Notice that this also implies both that $p\backslash x \le x$ and $x/p \le x$. We denote by $A^+$ the set of all positive elements of $\m A$, which one can readily prove is closed upwards, that is, if $p\in A^+$ and $p\le q$, then $q\in A^+$.

An element 1 of an ipo\nbd-semigroup $\m A$ is a \emph{global identity} if $1\cdot x = x$, for all $x\in A$. In particular, a global identity is always positive, and therefore $x\le x\cdot 1$. At the same time, $1\cdot\nein x = \nein x$ implies that $1\le \nein x/\nein x = x\backslash x$ and hence $x\cdot 1 \le x$, by residuation and contraposition. Thus, global identities also satisfy $x\cdot 1 = x$, for all $x\in A$, and therefore an ipo\nbd-semigroup cannot have more than one global identity. In that case, the structure $\widehat{\m A} = (A,\le,\cdot, 1, \nein, \no)$ is an ipo\nbd-monoid as defined in~\cite{GFJiLo23} and the set $A^+$ of positive elements of $\m A$ coincides with the \emph{positive cone} $\{p\in A : 1\le p\}$ of $\widehat{\m A}$. Indeed, an \emph{ipo\nbd-monoid} is a structure $(A,\le, \cdot, 1,\nein, \no)$ consisting of a poset $(A,\le)$, a monoid $(A,\cdot, 1)$, and two unary operations $\nein$ and $\no$ satisfying the following condition:
\[
x\cdot\nein y\le 0 \iff x\le y \iff \no y\cdot x \le 0, \tag{ineg}
\]
where $0=\no 1$. Then, if $\m A$ is an ipo\nbd-semigroup with a global identity $1$, we have that $(A,\le)$ is a poset and $(A,\cdot, 1)$ is a monoid. Also, $1\cdot\no 1\le \no 1$ implies $\no 1 = \no 1\cdot 1 = \no1\cdot\nein\no 1 \le\nein 1$, by~(dn) and~(rot). Analogously, $\nein 1\le \no 1$, and therefore $\no 1 = \nein 1$. Hence, rotation and the fact that $1\cdot x = x = x\cdot 1$ imply~(ineg). On the other hand, the 1-free reduct of any ipo\nbd-monoid is an ipo\nbd-semigroup.

\begin{lemma}
If $(A,\le,\cdot, 1, \nein, \no)$ is an ipo\nbd-monoid, then $(A,\le,\cdot,\nein,\no)$ is an ipo\nbd-semigroup.    
\end{lemma}

\begin{proof}
By definition, $(A,\le)$ is a poset and $(A,\cdot)$ is a semigroup. First of all, notice that (ineg) applied to $x\le x$ implies that $\no x\cdot x \le 0$ as well as $x\cdot\nein x\le 0$. In particular, $\no\nein x\cdot \nein x\le 0$ and $\no x\cdot \nein\no x\le 0$, which in turn imply that $\no\nein x\le x$ and $\nein\no x \le x$, respectively. From these inequations we obtain that $\no\nein\no x\le \no x$, from the first one, and that $\nein\no\nein\no x \le x$, from the second one, and then $\no x\cdot\nein\no\nein\no x \le 0$, and then $\no x \le \no\nein\no x$. That is, $\no\nein\no x = \no x$. In turn, this implies that $\no\nein\no x\cdot x = \no x\cdot x \le 0$, and therefore $x\le \nein\no x$. Thus, we have shown that $\nein\no x = x$. In particular $\nein\no\nein x = \nein x$, and hence $x\cdot\nein\no\nein x = x\cdot\nein x \le 0$, and then $x \le \no\nein x$. Thus, $\no\nein x = x$. This shows that $\nein$ and $\no$ satisfy double negation.

Now, using~(ineg), associativity, double negation, and~(ineg) again, we have:
\[
    xy \le z \iff (xy)\cdot\nein z \le 0
    \iff x(y\cdot\nein z) \le 0
    \iff \no\nein x(y\cdot\nein z) \le 0
    \iff y\cdot\nein z \le \nein x.
\]
And analogously, $xy\le z \iff \no z\cdot x\le \no y$. That is, the structure satisfies rotation. Finally, using rotation and the fact that $1$ is the identity of $\cdot$ we obtain that
\[
x \le y \iff x\cdot 1 \le y \iff 1\cdot\nein y \le \nein x \iff \nein y \le \nein x.
\]
And analogously, $x\le y \iff \no y \le \no x$. That is, the operations $\nein$ and $\no$ are antitone.
\end{proof}

We say that $\m A$ is \emph{integral} if it possesses a global identity $1$ which is also the top element of the order of $\m A$. The confluence of these two properties on $1$ imposes a number of restrictions on $\m A$. For instance, since $1\cdot x = x = x\cdot 1$, we obtain by residuation that $1 \le x/x$ and $1\le x\backslash x$, and since $1$ is the top element, we deduce that $x\backslash x = 1 = x/x$. Also, since $1\cdot x \le 1$ and $x\cdot 1\le 1$, we also have that $1 \le 1/x$ and $1 \le x\backslash 1$, and therefore $x\backslash 1 = 1 = 1/x$. The next lemma characterizes the integral ipo\nbd-semigroups as the ones satisfying two particular inequations. As a consequence, we obtain that the class of integral ipo\nbd-semigroups is a po-subvariety of the po-variety of ipo\nbd-semigroups, in the sense of~\cite{Pi04}. 

\begin{proposition}
An ipo\nbd-semigroup $\m A$ is integral if and only if it satisfies $x\le (x/x)x$ and $yx\le x$.
\end{proposition}

\begin{shortproof}
If $\m A$ is integral and $1$ is the global identity and top element of $\m A$, then for all $x\in A$, $1 = x/x$. In particular $x = 1\cdot x = (x/x)x$. Also, $yx\le 1\cdot x = x$. In the other direction, if $\m A$ satisfies $yx\le x$, by residuation we have that $y\le x/x$, that is, $x/x$ is the top element of $\m A$. Let's call this element $1$. Since by residuation we always have that $(x/x)x\le x$, we deduce that if in addition $\m A$ satisfies $x\le (x/x)x$, then we have that $1\cdot x = (x/x)x = x$, and therefore $\m A$ is integral.
\end{shortproof}

The following lemma about ipo\nbd-semigroups will be used numerous times in the following sections.

\begin{lemma}\label{lem:char:local:bounds}
For every ipo\nbd-semigroup $\m A$, the following conditions are equivalent:
\begin{enumerate}
    \item The identity $x\backslash x = x/x$ holds in $\m A$.
    \item The identity $\no x\cdot x = x\cdot\nein x$ holds in $\m A$,
\end{enumerate}
\end{lemma}

\begin{proof}
Suppose that the equation $x\backslash x = x/x$ holds in $\m A$, and in particular $\no x\backslash \no x = \no x/\no x$. Then, by double negation and contraposition,
\[
\no x\cdot x = \nein\no(\no x\cdot\nein\no x) = \nein (\no x/\no x) = \nein (\no x\backslash\no x) = \nein (x/x)
= \nein\no(x\cdot\nein x) = x\cdot\nein x.
\]
In order to prove the other implication, suppose that the equation $\no x\cdot x = x\cdot\nein x$ holds in $\m A$, and in particular $\no\nein x\cdot \nein x=\nein x\cdot\nein \nein x$. Then, by contraposition and double negation,
\[
x\backslash x = \nein x/\nein x = \no (\nein x\cdot\nein\nein x) = \no (\no\nein x\cdot\nein x) = \no(x\cdot \nein x) = x/x.\tag*{\QED}
\]\def\QED{}
\end{proof}

We end this section by underlining a special class of ipo\nbd-semigroups, namely those which have a lattice order. These can be then presented as algebraic structures $(A,\land,\lor,\cdot,\nein,\no)$ called \emph{i$\ell$-semigroups} or \emph{involutive semirings}\footnote{~This terminology is based on the observation that $(A,\lor,\cdot)$ is an idempotent semiring since the residuation property implies that $x(y\lor z)=xy\lor xz$ and $(x\lor y)z=xz\lor yz$, and $\land$ is term definable by the De~Morgan laws.} and \emph{integral i$\ell$-semigroups} or \emph{integral involutive semirings}, if they are integral as ipo\nbd-semigroups. Notice that, since the inequality $xy \le y$ can be expressed as $xy \lor y = y$ in the language of involutive semirings, the integral involutive semirings form a subvariety of the variety of involutive semirings.

Note that the variety of unital i$\ell$-semirings (without the assumption of integrality) coincides with the variety of involutive residuated lattices, and unital Frobenius quantales are the same as complete unital i$\ell$-semirings.

\section{Locally integral ipo-semigroups and involutive semirings}
\label{sec:locally:integral:ipo-semigroups}

An ipo\nbd-semigroup $\m A$ has \emph{local identities} if for all $x\in A$, $x\backslash x = x/x$, in which case we use $1_x$ to denote this element, and $1_x\cdot x = x$. Notice that, if $\m A$ has local identities, by Lemma~\ref{lem:char:local:bounds}, we also have that for all $x\in A$, $\no x\cdot x = x\cdot\nein x$ and we use $0_x$ to denote this element. If $\m A$ possesses a global identity $1$, then it is the smallest of the local identities: this is because $1\backslash 1 = 1 = 1/1$ and then $1_1 = 1$, and $1\cdot x = x$ implies that $1\le x/x = 1_x$, for all $x\in A$. Hence, if $\m A$ has a global identity, all local identities are positive. The following are other fundamental properties of local identities.

\begin{lemma}\label{lem:local:identities}
If $\m A$ has local identities, then $\m A$ satisfies the following equations:
\begin{enumerate}
    \item $1_x = \no 0_x = \nein 0_x$ and $0_x = \no 1_x = \nein 1_x$.
    \item\label{lem:part:bounds:and:negations} $1_x = 1_{\no x} = 1_{\nein x}$ and $0_x = 0_{\no x} = 0_{\nein x}$.
    \item $ x\backslash 0_x = \nein x$ and $x/1_x = x$.
    \item $1_x1_x = 1_x$ and $x\cdot 1_x = x$.
    \item $0_x/x = \no x$ and $1_x\backslash x = x$.
\end{enumerate}
\end{lemma}

\eject

\vspace*{-6ex}

\begin{shortproof}
\vspace{-1ex}
\begin{enumerate}
    \item Notice that $1_x = x/x = -(x\cdot \nein x) = -0_x$ and $0_x = \nein\no 0_x = \nein 1_x$. The other two equalities are proven in the same way.
    
    \item We have that $0_{\nein x} = \no\nein x\cdot \nein x = x\cdot \nein x = 0_x$, and therefore $1_{\nein x} = \no 0_{\nein x} = \no 0_x = 1_x$. The proof that $0_{\no x} = 0_x$ and $1_{\no x} = 1_x$ is analogous.

    \item The first equality follows from $x\backslash 0_x = \nein(\no 0_x\cdot x) = \nein(1_x\cdot x) = \nein x$. As for the second, $x/1_x = -x\backslash -1_x = -x\backslash 0_x = \nein\no x = x$.

    \item First, notice that $0_x1_x = 0_x\cdot\nein 0_x = \no 0_x\cdot 0_x = 1_x0_x = 1_x\cdot x\cdot\nein x = x\cdot\nein x = 0_x$. Hence, $1_{1_x} = 1_x/1_x = \no(1_x\cdot\nein 1_x) = \no(1_x0_x) = \no 0_x = 1_x$. And therefore, $1_x1_x = 1_{1_x}\cdot 1_x = 1_x$, as we wanted to show. Now, $\no(x1_x) = \no(x1_x1_x) = \no(x1_x\cdot\nein 0_x) = 0_x/(x1_x) = (0_x1_x)/(x1_x) = (-x(x1_x))/(x1_x)$ and multiplying by $x1_x$ we obtain
    \[
    0_{x1_x} = -(x1_x)(x1_x) = \big((-x(x1_x))/(x1_x)\big)(x1_x) = -x(x1_x) = 0_x1_x = 0_x.
    \]
    Thus, $1_{x1_x} = \no 0_{x1_x} = \no 0_x = 1_x$, whence we deduce that
    \[
    x \le (x1_x)/1_x = (x1_x)/(1_{x1_x}) = x1_x = x(x\backslash x) \le x,
    \]
    and therefore $x1_x = x$.

    \item The first equality follows from $0_x/x = \no(x\cdot\nein 0_x) = \no (x\cdot 1_x) = \no x$ and for the second equality, $1_x\backslash x = \nein 1_x/\nein x = 0_x/\nein x = \no\nein x = x$.\QED
\end{enumerate}\def\QED{}
\end{shortproof}

\begin{remark}\label{rem:loc:ident:commutat:implies:cyclic}
Commutativity doesn't imply cyclicity for ipo\nbd-semigroups (see Figure~\ref{fig:commutative:noncyclic}). But, it follows from the previous lemma that it does for ipo\nbd-semigroups with local identities, since in any commutative ipo\nbd-semigroup the equality $y/x = x\backslash y$ holds and therefore $\no x = 0_x/x = x\backslash 0_x = \nein x$. Hence, the example of Figure~\ref{fig:commutative:noncyclic} doesn't have local identities. Indeed, $1_a\cdot a = \bot\neq a$.
\end{remark}

\begin{figure}
\centering\scriptsize
\begin{tikzpicture}[xscale=.8,yscale=.7]
\node at (4,1)[n]{$x\cdot y = \bot$};
\node at (8,4)[n]{%
\begin{tabular}{>{$}c<{$}|>{$}c<{$}>{$}c<{$}>{$}c<{$}>{$}c<{$}>{$}c<{$}}
        & \top & a & b & c & \bot\\\hline
\nein   & \bot & b & c & a & \top \\
\no     & \bot & c & a & b & \top 
\end{tabular}};
\draw 
(-6,1)node[label=below:$\bot$]{}--
(-4,3)node[label=right:$c$]{}--
(-6,5)node[label=above:$\top$]{}--
(-8,3)node[label=left:$a$]{}--(-6,1)--
(-6,3)node[label=right:$b$]{}--(-6,5);
\end{tikzpicture}
\caption{Commutative and noncyclic ipo\nbd-semigroup.}
\label{fig:commutative:noncyclic}
\end{figure}
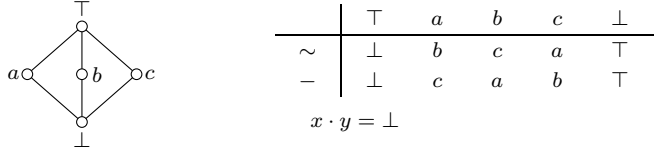

Inspired by the notion of an integral ipo\nbd-semigroup and its elementary properties, we define an ipo\nbd-semigroup $\m A$ to be \emph{locally integral} if it has local identities ($1_x = x\backslash x = x/x$ and $1_x\cdot x = x$) satisfying, moreover, the following conditions for all $x,y\in A$:
\begin{enumerate}
    \item the local identities are positive, that is, $y\le 1_xy$,
    \item the elements are \emph{locally bounded} by their identities, that is, $x\le 1_x$,
    \item and $x\backslash 1_x = 1_x$.
\end{enumerate}
As we have seen before, in an integral ipo\nbd-semigroup with global identity $1$ we have that $x\backslash x = 1 = x/x$, hence $1_x = 1$ and $1_x\cdot x = 1\cdot x = x$. Thus, all local identities coincide with 1, which is positive and also satisfies $x\le 1$ and $x\backslash 1 = 1$. That is, every integral ipo\nbd-semigroup is locally integral. The following proposition characterizes local integrality for ipo\nbd-semigroups and shows that if a locally integral ipo\nbd-semigroup $\m A$ has a global identity, then the corresponding $\widehat{\m A}$ is a locally integral ipo\nbd-monoid as defined in~\cite{GFJiLo23}, and the 1-free reducts of the locally integral ipo\nbd-monoids are locally integral ipo\nbd-semigroups.

\begin{proposition}\label{prop:char:locally:integral}
An ipo\nbd-semigroup $\m A$ is locally integral if and only if it satisfies the following conditions for all $x,y\in A$:
\begin{enumerate}
    \item $-x\cdot x = x\cdot\nein x$,
    \item $x/x$ is positive, that is, $y \le (x/x)y$,
    \item multiplication is \emph{square decreasing}, that is, $xx\le x$,
    \item $0_x$ is idempotent, that is, $0_x0_x = 0_x$.
\end{enumerate}
\end{proposition}

\begin{proof}
By virtue of Lemma~\ref{lem:char:local:bounds}, condition~(1) is equivalent to $x\backslash x = x/x$, and together with condition~(2), we have that $x \le (x/x) x \le x$, that is, $1_x\cdot x = x$, and hence $\m A$ has local identities. Condition~(2) corresponds to the positivity of the local identities, and being square decreasing is equivalent to $x\le 1_x$ by residuation. As for the last condition, if $x\backslash 1_x = 1_x$ holds in $\m A$, then $0_x = \no(x\backslash 1_x) = \no\nein(\no 1_x\cdot x) = 0_x\cdot x$. Therefore,
\[
0_x\cdot 0_x = 0_x \cdot x\cdot\nein x = 0_x\cdot\nein x = 0_{\nein x}\cdot\nein x = 0_{\nein x} = 0_x.
\]
For the converse, if $0_x$ is idempotent in $\m A$, then
\[
0_x = 0_x0_x = 0_x\cdot x\cdot\nein x \le 0_x\cdot x\cdot 1_{\nein x} = 0_x\cdot x\cdot 1_x = 0_x\cdot x \le 0_x\cdot 1_x = 0_x,
\]
since $\nein x\le 1_{\nein x} = 1_x$ and $x\le 1_x$. Hence $0_x\cdot x = 0_x$, whence $1_x = \nein 0_x = \nein(0_x\cdot x) = x\backslash 1_x$.
\end{proof}

\begin{example}
Consider the structure $\m A =(\{e,a\},=,\cdot,\nein,\no)$, with the commutative multiplication given by $e\cdot x = x$ and $a\cdot a = e$, and the involutive negations given by $\nein x = \no x = x$. We can readily check that this is a commutative ipo\nbd-semigroup and moreover $x\backslash y = y/x = x\cdot y$, whence we deduce that $1_x= x\cdot x = e$ and thus $1_x\cdot x = e\cdot x = x$, for all $x$. In particular, $\m A$ has local identities, both of them equal to $e$. Nonetheless, $\m A$ is not locally integral, because it is not square decreasing.
\end{example}

\begin{remark}
If a locally integral ipo\nbd-semigroup has a minimal (or maximal) element then it is bounded. In particular all finite ipo\nbd-semigroups are bounded. However, there exist unbounded idempotent locally integral ipo\nbd-monoids.    
\end{remark}

\begin{example}\label{ex:loc:integ:no:global:unit}
Consider the cyclic, commutative, and idempotent ipo\nbd-semigroup of Figure~\ref{fig:4-element:without:global:identity}, whose product is determined by $\bot\cdot x = \bot$ and $p\cdot q = p\cdot\top = q\cdot \top = \top$, and the involutive negations are the identity on $p$ and $q$ and transpose $\top$ and $\bot$. This ipo\nbd-semigroup has local identities $1_p =p$, $1_q=q$, and $1_\top=\top$, and by definition it satisfies all the properties of Proposition~\ref{prop:char:locally:integral}. Thus, it is locally integral. It doesn't contain a global identity, and therefore it is not integral. Notice also that $0_p \nle 1_q$ and $0_q\nle 1_p$.
\end{example}

\begin{figure}
\centering\scriptsize
\begin{tikzpicture}[xscale=.8,yscale=.7]
\draw 
(-6,1)node[label=below:$\bot$]{}--
(-4,3)node[label=right:{\,$q=1_q=0_q$}]{}--
(-6,5)node[label=above:$\top$]{}--
(-8,3)node[label=left:{$0_p=1_p=p$\,}]{}--cycle;
\end{tikzpicture}
\caption{Locally integral ipo\nbd-semigroup with three different local identities.}
\label{fig:4-element:without:global:identity}
\end{figure}
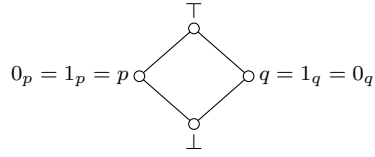

The main goal of this section is a decomposition theorem stating that every locally integral ipo\nbd-semigroup (involutive semiring, respectively) can be decomposed in a very particular way into integral involutive ipo\nbd-semigroups (involutive semirings, respectively). But, before addressing this decomposition, we need a better description of the positive elements of a locally integral ipo\nbd-semigroup~$\m A$. Recall that $A^+ = \{p\in A : \text{for all }x, x\le px\} = \{p\in A : \text{for all }x, x\le xp\}$ and that by definition $1_x\in A^+$, for all $x\in A$. We also define the set $A^\sim = \{\nein p : p\in A^+\}$, which is closed downwards, since $A^+$ is closed upwards and $\nein$ is antitone.

\begin{proposition}\label{prop:positive:are:central}
The positive elements of any locally integral ipo\nbd-semigroup are central.
\end{proposition}

\begin{proof}
Suppose that $p$ is a positive element and let $x$ be an arbitrary element. By the square decreasing property, $pp\le p$ and therefore $ppx\le px$, which implies $p\le px/px=px\backslash px$ by residuation and local integrality. Therefore $pxp\le px$. Since $p$ is positive, we have $x\le px$, and multiplying by $p$, we obtain $xp\le pxp\le px$. The reverse inequality follows by symmetry.
\end{proof}

\begin{lemma}\label{lem:char:positives:and:zeros}
Let $\m A$ be a locally integral ipo\nbd-semigroup. For all $p$ and $a$ in $A$, we have that $p\in A^+$ if and only if $1_p = p$ and $a\in A^\sim$ if and only if $0_a = a$. As a consequence, $A^+ = \{1_x : x\in A\}$ and $A^\sim = \{0_x : x\in A\}$. Moreover, $0_x \le 1_x$ for every $x\in A$, that is, $\nein p \le p$, for every $p\in A^+$.
\end{lemma}

\begin{proof}
By local integrality, $p \le 1_p$ is valid for all $p$ in $A$. If moreover $p\in A^+$, then $1_p \le p\cdot 1_p = p$ and hence, $1_p = p$. The reverse implication is trivial, since by local integrality $1_p\in A^+$. For the second part, suppose that $a\in A^\sim$ and let $p\in A^+$ so that $a = \nein p$. Then,
\[
a= \nein p = \nein 1_p = 0_p = 0_{\nein p} = 0_a.
\]
And for the reverse implication, notice that $0_a = \nein 1_a\in A^\sim$, since $1_a\in A^+$ by local integrality. Finally, for every $x\in A$ we have that $0_x \le 1_{0_x} = 0_x/0_x = -0_x = 1_x$.
\end{proof}

Given a locally integral ipo\nbd-semigroup $\m A$, the sets $A^+$ and $A^\sim$ are obviously partially ordered by the order of $\m A$. It turns out that $(A^+,\le)$ and $(A^\sim,\le)$ are semilattices dual to each other, and their properties are very consequential. For instance, if $(A^+,\le)$ has a minimum element $p$ then it is the global identity of $\m A$, because for every $a\in A$ we have that $p\le 1_a = a/a$, and therefore $pa\le a$; but since $p$ is positive, we also have that $a\le pa$, showing that $pa=a$. The next proposition describes these two posets in more detail.

\begin{proposition}\label{prop:A+:join:semilattice}
Let $\m A$ be a locally integral ipo\nbd-semigroup. Then $\m A^+ = (A^+,\cdot)$ is a join-semilattice whose order coincides with the order of $\m A$. Also, $\m A^\sim = (A^\sim, +)$ is a meet-semilattice, whose order coincides with the order of $\m A$, and is dually isomorphic to $\m A^+$. Moreover, if $x,y\le a\in A^\sim$, then $xy = x\land y = x+y$. As a consequence, if $X\subset A^+$ is closed under joins and meets, then $(X,\le)$ is a distributive lattice.
\end{proposition}

\begin{proof}
For the first part, it is immediate to check that $A^+$ is closed under products. Furthermore, for all $p,q\in A^+$, we have that that $p \le pq$ and $q \le pq$. Moreover, if $p \le r$ and $q \le r$, then $pq \le rr \le r$. This shows that $pq = p\lor q$ and therefore $\m A^+ = (A^+,\cdot)$ is a join semilattice whose induced order is the restriction of $\le$ to $A^+$. The second part is a consequence of the fact that the order-reversing bijection $\eta\: A^+\to A^\sim$ given by $\eta(p) = \nein p$ satisfies $\eta(p\cdot q) = \nein (p\cdot q) = \nein (\no\nein p\cdot\no\nein q) =  \nein p + \nein q = \eta(p) + \eta(q)$. Hence, $x+y = x\land y$ for all $x,y\in A^\sim$. Using the characterization of Lemma~\ref{lem:char:positives:and:zeros}, for all $p,q\in A^+$,
\[
0_p\land 0_q =0_p+0_q = \nein(\no 0_p\cdot \no 0_q) = \nein(1_p\cdot 1_q) = \nein (pq) = \nein 1_{pq} = 0_{pq}.
\]

Suppose now that $a\in A^\sim$ and $x,y\le a$. Then, there are $p,q,r\in A^+$ such that $x = \nein p$, $y=\nein q$, $a = \nein r$, and moreover $r\le p, q$. Notice that $x\le a =\nein r \le r \le q$ and hence, $xy\le qy = q\cdot\nein q = 0_q = \nein 1_q = \nein q = y$. Analogously, $xy\le x$ and therefore $xy\le x\land y = 0_p\land 0_q = 0_{pq} = 0_{pq}0_{pq} \le xy$, since $\no x = p\le p\lor q = pq = 1_{pq}$ and also $\no y\le 1_{pq}$. Hence, $xy = x\land y = x+y$.

Finally, if $X\subset A^+$ is closed under joins and meets, then so is $X^\sim = \{\nein p : p\in X\}$, and for all $x,y,z\in X^\sim$ we have that
\[
x\land(y\lor z ) = x\cdot(y\lor z) = (x\cdot y) \lor (x\cdot z) = (x\land y) \lor (x\land z).
\]
Thus, $(X^\sim,\le)$ is a distributive lattice, and therefore so is the dually-isomorphic lattice $(X,\le)$.
\end{proof}

In order to obtain the aforementioned decomposition theorem, we start by noticing that the equivalence relation $x \equiv  y$ if and only if $1_x = 1_y$ partitions every locally integral ipo\nbd-semigroup into its equivalence classes $A_x = \{y\in A : 1_x = 1_y\}$ and, obviously, $x\in A_x$. The next lemma offers a very useful description of $A_x$ in which we make use of the interval notation $[0_x,1_x] = \{y\in A : 0_x \le y \le 1_x\}$. Recall that by virtue of Lemma~\ref{lem:local:identities} and Proposition~\ref{prop:char:locally:integral}, both $1_x$ and $0_x$ are idempotent. This implies that every interval $[0_x,1_x]$ is closed under products. Indeed, if $y,z\in [0_x,1_x]$ then $0_x = 0_x0_x \le yz \le 1_x1_x = 1_x$, by monotonicity, and hence $yz\in [0_x,1_x]$.

\eject

\begin{lemma}\label{lem:char:Ax}
For any locally integral ipo\nbd-semigroup $\m A$ and all $x$ and $y$ in $A$:
\begin{enumerate}
    \item $x\in [0_x,1_x]$,
    \item $y\in [0_x, 1_x] \iff 1_y\le 1_x \iff [0_y,1_y]\subset [0_x, 1_x]$,
    \item $1_x\cdot y = y \iff 1_x\le 1_y$,
    \item\label{lem:part:char:Ax} $y\in A_x \iff y\in [0_x, 1_x]$ and $1_x\cdot y = y$.
\end{enumerate}
\end{lemma}

\begin{proof}\vspace{-\baselineskip}
\begin{enumerate}
    \item  By antitonicity, $\nein x\le 1_{\nein x}$ implies that $0_x = 0_{\nein x} = \no 1_{\nein x} \le \no\nein x = x$. Hence, $x\in [0_x, 1_x]$.
    
    \item For the left-to-right implication, notice that $0_x\le y\le 1_x$ implies that $0_x\le \nein y\le  1_x$, by antitonicity, and then $0_x = 0_x\cdot 0_x \le y\cdot \nein y = 0_y$. By antitonicity again, we obtain that $1_y\le 1_x$ and hence $[0_y,1_y]\subset [0_x,1_x]$. The reverse implication is a consequence of part~(1).
    
    \item Since $1_x\in A^+$, we have that $y \le 1_x\cdot y$. Hence, $1_x\cdot y = y$ is equivalent to $1_x\cdot y \le y$ which is equivalent to $1_x \le y/y = 1_y$, by residuation.
    
    \item If $y\in A_x$ then $1_y = 1_x$, and thus $y\in [0_y,1_y] = [0_x, 1_x]$, by part~(1). Moreover, $1_x\cdot y = 1_y\cdot y = y$. For the reverse implication, notice that if $y\in [0_x,1_x]$ and $1_x\cdot y = y$, then $1_y \le 1_x$, by part~(2), and $1_x \le 1_y$, by part~(3).\QED
\end{enumerate}\def\QED{}
\end{proof}

Next, we will use the description of $A_x$ of the previous lemma in order to show that the sets $A_x$ are closed under several operations of $\m A$.

\begin{lemma}\label{lem:Ax:closed:negations:mult:meets:joins}
Let $\m A$ be a locally integral ipo\nbd-semigroup. For every $x$ in $A$:
\begin{enumerate}
    \item $A_x$ is closed under the involutive negations,
    \item $A_x$ is closed under multiplication,
    \item $A_x$ is closed under all existing nonempty joins and nonempty meets.
\end{enumerate}
\end{lemma}

\begin{proof}\vspace{-\baselineskip}
\begin{enumerate}
    \item If $y\in A_x$ then $1_{\nein y} = 1_y = 1_x$, and hence $\nein y\in A_x$.
    
    \item If $y,z\in A_x$ then $y,z \in [0_x,1_x]$ and $1_x\cdot y=y$ and $1_x\cdot z = z$. Hence, $yz\in [0_x,1_x]$ and $1_x\cdot (yz) = (1_x\cdot y)z = yz$. Therefore, $yz\in A_x$, by Lemma~\ref{lem:char:Ax}.

    \item Suppose that $\emptyset \neq Y\subset A_x$ and the join $\bigvee Y$ exists in $A$. Since for every $y$ in $Y$, $y\in A_x \subset [0_x,1_x]$, we obtain that also $\bigvee Y \in [0_x,1_x]$. And since multiplication distributes with respect to all existing joins, we have that $1_x\cdot \bigvee Y = \bigvee_{y\in Y} 1_x\cdot y = \bigvee_{y\in Y} y = \bigvee Y$. Thus, by Lemma~\ref{lem:char:Ax}, $\bigvee Y \in A_x$. The closure under all existing nonempty meets can be obtained from the fact that $A_x$ is also closed under negations and $\bigwedge Y = \no\bigvee_{y\in Y} \nein y$.\QED
\end{enumerate}\def\QED{}
\end{proof}

Our next goal is to find a canonical representative for each equivalence class~$A_x$. 

\begin{lemma}\label{lem:char:bounds:integral:components}
Let $\m A$ be a locally integral ipo\nbd-semigroup. For every $x$ in $A$, $1_x$ is the only positive element of $A_x$ and $0_x$ is the only element of $A_x\cap A^\sim$.
\end{lemma}

\begin{shortproof}
Since $1_x\in [0_x,1_x]$ and also $1_x\cdot 1_x = 1_x$, we have that $1_x \in A_x$, by Lemma~\ref{lem:char:Ax}. Now, for any positive $p\in A_x$, we have that $p = 1_p = 1_x$, by Lemma~\ref{lem:char:positives:and:zeros}. The second part follows from Lemma~\ref{lem:char:positives:and:zeros}, antitonicity, and the fact that $A_x$ is closed under the involutive negations.
\end{shortproof}

\begin{remark}\label{rem:Ap:partition:of:A}
Notice that the previous lemma tells us that for every $x$ in $A$, there is only one positive element $p$ such that $A_x = A_p$. This means that the family $\{A_x : x\in A\}$ is actually indexed by $A^+$ and that for all $p,q\in A^+$, we have $A_p = A_q$ if and only if $p=q$. 
\end{remark}

We can now show that the relation and operations of a locally integral ipo\nbd-semigroup furnish each equivalence class $A_p$ with the structure of an integral ipo\nbd-monoid $\m A_p$. We call every $\m A_p$ an \emph{integral component} of $\m A$.

\begin{proposition}\label{prop:integral:components}
If $\m A$ is a locally integral ipo\nbd-semigroup, then for every $p\in A^+$, the structure $\m A_p = (A_p, \le, \cdot, 1_p, \nein, \no)$, where the relation and the operations are the restrictions to $A_p$ of the corresponding relation and operations of $\m A$, is an integral ipo\nbd-monoid. If in addition $\m A$ is a semiring, cyclic, or commutative, then every $\m A_p$ is also a semiring, cyclic, or commutative, respectively.
\end{proposition}

\begin{proof}
By Lemma~\ref{lem:Ax:closed:negations:mult:meets:joins}, every $A_p$ is closed under multiplication and involutive negations, and by Lemma~\ref{lem:char:positives:and:zeros}, $1_p = p \in A_p$. Therefore, the structure $\m A_p$ is well defined and the reduct $(A_p, \le, \cdot, \nein, \no)$ is an ipo\nbd-semigroup. Moreover, $1_p\cdot x = 1_x\cdot x = x$ for all $x\in A_p$ and by Lemma~\ref{lem:char:Ax}, $A_p \subset [0_p,1_p]$, and consequently $1_p$ is the top element of $(A,\le)$. Hence, $(A_p, \le, \cdot, \nein, \no)$ is integral and $1_p$ is its global identity, rendering $\m A_p$ an integral ipo\nbd-monoid. Obviously, cyclicity and commutativity are hereditary properties. Finally, the proof for the locally integral involutive semirings follows from the fact that $A_p$ is also closed under all existing binary joins and meets, by Lemma~\ref{lem:Ax:closed:negations:mult:meets:joins}.
\end{proof}

\begin{figure}
\centering
\begin{tikzpicture}[scale=1.25, every label/.append style={font=\scriptsize}]
\node at (0,3)[n]{\scriptsize$\m A_p$};
\node at (-.6,.9)[n]{\scriptsize$A^\sim$};
\node at (-.7,5.1)[n]{\scriptsize$A^+$};
\node at (1.3,5.5)[n]{\scriptsize$\m A$};

\draw (0,0)..controls(-3,2)and(-3,4)..(0,6)..controls(3,4)and(3,2)..(0,0)
(-2,4)--(-1.4,4.5)--(-.8,4.2)--(0,4.7)--(.8,4.2)--(1.4,4.5)--(2,4)
(-2,2)--(-1.4,1.5)--(-.8,1.8)--(0,1.3)--(.8,1.8)--(1.4,1.5)--(2,2)
(0,1) node(0p)[label=-45:$0_p$]
{}..controls(-1.33,2.33)and(-1.33,3.66)..(0,5) node(1p)[label=45:$1_p$]{}..controls(1.33,3.66)and(1.33,2.33)..(0,1);
\end{tikzpicture}
\qquad\qquad\qquad
\begin{tikzpicture}[scale=1.25, every label/.append style={font=\scriptsize}]
\node at (0,3)[n]{\scriptsize$\m A_p$};
\node at (-.6,.9)[n]{\scriptsize$A^\sim$};
\node at (-.7,5.1)[n]{\scriptsize$A^+$};
\node at (-1.8,3)[n]{\scriptsize$\m A_1$};
\node at (1.3,5.5)[n]{\scriptsize$\m A$};

\draw (0,0)..controls(-3,2)and(-3,4)..(0,6)..controls(3,4)and(3,2)..(0,0)
(-2,4) node(1)[label=left:$1$]{}..controls(-1,5)and(0.75,5.1)..(1.15,5.1)
(-2,2) node(0)[label=left:$0$]{}..controls(-1,1)and(0.75,0.9)..(1.15,0.9)
(1)..controls(-1.25,3.5)and(-1.25,2.5)..(0)
(0,1) node(0p)[label=-45:$0_p$]{}..controls(-1.33,2.33)and(-1.33,3.66)..(0,5) node(1p)[label=45:$1_p$]{}..controls(1.33,3.66)and(1.33,2.33)..(0,1);
\end{tikzpicture}
\caption{Representation of the structure of a locally integral ipo\nbd-semigroup and ipo-monoid.}
\label{fig:schema:structure:ipomonoid}
\end{figure}
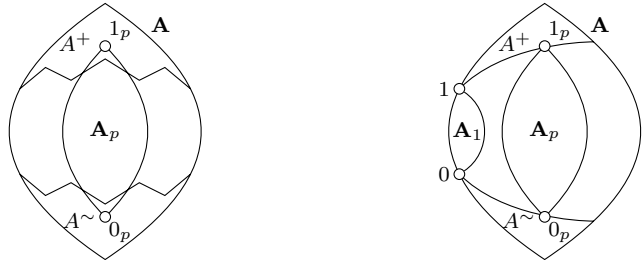

What we have proven so far is that every locally ipo\nbd-semigroup $\m A$ can be broken into a disjoint union of a family of ipo\nbd-monoids $\{\m A_p : p\in A^+\}$. But, obviously, $\m A$ is not the mere ``union'' of these ipo\nbd-monoids: the elements from different components can be multiplied and compared with each other. This is the content of the following two lemmas.

\begin{lemma}\label{lem:prod:in:Aprod}
Given a locally integral ipo\nbd-semigroup $\m A$, positive elements $p$ and $q$, and elements $x\in A_p$ and $y\in A_q$, the product $xy$ is in $A_{pq}$.
\end{lemma}

\begin{proof}
The inequalities $1_p = p \le pq = 1_{pq}$ and $1_q \le 1_{pq}$ imply that $A_p\cup A_q \subset [0_p,1_p]\cup [0_q,1_q] \subset [0_{pq}, 1_{pq}]$, and therefore $x,y\in [0_{pq}, 1_{pq}]$, whence we deduce that $xy\in  [0_{pq}, 1_{pq}]$. Moreover, $1_{pq}\cdot(xy) = pqxy = pxqy = 1_x\cdot x\cdot 1_y\cdot y=xy$, by Proposition~\ref{prop:positive:are:central}. Hence, by Lemma~\ref{lem:char:Ax}, $xy\in A_{pq}$.
\end{proof}

\begin{lemma}\label{lem:A:is:balanced}
Let $\m A$ be a locally integral ipo\nbd-semigroup and $p$ and $q$ positive elements. Then for all $x \in A_p$ and $y \in A_q$,
\[
x\le y \iff x \cdot \nein y = 0_{pq} \iff \no y \cdot x = 0_{pq}.
\]
\end{lemma}

\begin{proof}
By Lemmas~\ref{lem:Ax:closed:negations:mult:meets:joins}, \ref{lem:prod:in:Aprod}, and~\ref{lem:char:Ax}, both $x \cdot \nein y$ and $\no y \cdot x$ are in $A_{pq} \subseteq [0_{pq},1_{pq}]$. Hence $0_{pq} \le x \cdot \nein y$ and $0_{pq} \le \no y \cdot x$. It therefore suffices to prove that $x\le y \iff x \cdot \nein y \le 0_{pq} \iff \no y \cdot x \le 0_{pq}$. For the first equivalence, using~(rot) and Proposition~\ref{prop:positive:are:central} we obtain that
\begin{align*}
x\le y &\iff 1_p x\le y \iff x\cdot\nein y \le 0_p \iff x 1_q\cdot\nein y \le 0_p \\
&\iff 1_p x 1_q \le y \iff 1_p1_qx \le y \iff 1_{pq}x\le y \iff x\cdot\nein y \le 0_{pq}.
\end{align*}
Also, $x \cdot \nein y \le 0_{pq} \iff 1_{pq} \cdot x \le y \iff x \cdot 1_{pq} \le y \iff \no y \cdot x \le 0_{pq}$.
\end{proof}

All these results point toward the idea that locally integral ipo\nbd-semigroups are built up from integral ipo\nbd-monoids, or at least their semigroup reducts are, by means of a P\l{}onka sum. This construction was first introduced and studied in~\cite{Plonka67,Plonka68a,Plonka68b}; for more recent expositions see~\cite{PR92} and~\cite{BPP22}. 
Let $(I, \lor)$ be a join-semilattice. A family of homomorphisms $\{\varphi_{ij}\: \m A_i\to \m A_j : i\le j \in I\}$ between algebras of the same type is said to be \emph{compatible} if for every $i\in I$, $\varphi_{ii}$ is the identity on $\m A_i$, and that if $i \le j \le k$ then $\varphi_{jk}\circ\varphi_{ij} = \varphi_{ik}$. Given such a compatible family of homomorphisms, its \emph{P\l{}onka sum} is an algebra $\m S$ defined on the disjoint union of their universes $S = \biguplus_{i\in I} A_i$. For every nonconstant $n$-ary operation symbol $\sigma$ and elements $a_1\in A_{i_1}$, \dots, $a_n\in A_{i_n}$, $\sigma^{\m S}(a_1,\dots,a_n) = \sigma^{\m A_j}(\varphi_{i_1j}(a_1),\dots,\varphi_{i_nj}(a_n))$, where $j = i_1\lor \dots\lor i_n$. One can readily prove that the P\l{}onka sum of a compatible family of homomorphisms is well-defined and it satisfies all regular equations that hold in all the algebras of the family which do not involve constants. Recall that a \emph{regular} equation is an equation in which the variables that appear on the left-hand side are the same as the variables that appear on the right-hand side.

Given a locally integral ipo\nbd-semigroup $\m A$, consider the map $\varphi_{pq}\: A_p \to A_q$ given by $\varphi_{pq}(x) = qx$, for every pair of positive elements $p \le q$. We will show that $\Phi = \{\varphi_{pq} : p\le q\}$ is a compatible family of monoid homomorphisms between the family of algebras $\{\m A_p : p\in A^+\}$, and that the semigroup reduct of $\m A$ can be reconstructed as the P\l{}onka sum of $\Phi$.\footnote{~The general theory (\cite{BPP22}) tells us that P\l{}onka sums can be described using partition functions. We note that, in this particular setting, the partition function is defined by $a\odot b = 1_b\cdot a$.}

\begin{lemma}
Let $\m A$ be a locally integral ipo\nbd-semigroup and $p \le q$ two positive elements. Then $\varphi_{pq}\: \m A_p \to \m A_q$ is a well defined monoid homomorphism. Moreover, it respects arbitrary nonempty existing joins and is therefore monotone.
\end{lemma}

\begin{proof}
For all positive elements $p$ and $q$, and $x\in A_p$, we have that $qx \in A_{qp}$, by Lemma~\ref{lem:prod:in:Aprod}. Moreover, by Proposition~\ref{prop:A+:join:semilattice}, the inequality $p \le q$ implies that $qp = q$. Hence, the map $\varphi_{pq}\: A_p \to A_q$ is well defined. Furthermore, $\varphi_{pq}(1_p) = q1_p = qp = q = 1_q$ and for all $x,y\in A_p$,
\[
\varphi_{pq}(x\cdot y) = qxy = qqxy = qxqy = \varphi_{pq}(x)\cdot\varphi_{pq}(y),
\]
since $q$ is positive and therefore idempotent and central, by Proposition~\ref{prop:positive:are:central}. This shows that $\varphi_{pq}$ is a monoid homomorphism. Finally, if $\emptyset \neq Y\subset A_p$ is such that $\bigvee Y$ exists, then $\varphi_{pq}\big(\bigvee Y\big) = q\cdot \bigvee Y = \bigvee_{y\in Y} qy = \bigvee_{y\in Y} \varphi_{pq}(y)$.
\end{proof}

\begin{proposition}
Let $\m A$ be a locally integral ipo\nbd-semigroup. Then, its associated family $\Phi = \{\varphi_{pq}\:$ $\m A_p\to \m A_q : p\le q\}$ is a compatible family of monoid homomorphisms indexed by the order of the join semilattice $\m A^+$. 
\end{proposition}

\begin{proof}
For every positive element $p$ and $x\in A_p$, we have $\varphi_{pp}(x) = px = 1_p x = x$, since $x\in A_p$. That is, $\varphi_{pp}$ is the identity homomorphism on $\m A_p$. And if $p \le q \le r$ are positive elements, then $\varphi_{qr}(\varphi_{pq}(x)) = rqx = rx = \varphi_{pr}(x)$, since $rq = r$ by Proposition~\ref{prop:A+:join:semilattice}, because $q \le r$.
\end{proof}

As we will show next, the semigroup reduct of a locally integral ipo\nbd-semigroup is the P\l{}onka sum of (the semigroup reducts of) the family above. Moreover, although this is not the case for the rest of the structure, we can still recover it from its integral components by virtue of Lemma~\ref{lem:A:is:balanced}. Also, we saw in Proposition~\ref{prop:integral:components} that some properties of $\m A$ are inherited by each of its integral components. Sometimes the converse is also true. We call a property of ipo\nbd-semigroups \emph{local} whenever an ipo\nbd-semigroup has it if and only if all its integral components have it.

\begin{theorem}\label{thm:decomposition:ipo-monoids}
Let $\m A$ be a locally integral ipo\nbd-semigroup and $\Phi$ its associated family of monoid homomorphisms defined above. Then, its P\l{}onka sum $\big(\biguplus_{p\in A^+} A_p, \cdot^{\m S}\big)$ is the semigroup reduct of $\m A$, where $x\cdot^{\m S} y=\varphi_{pr}(x)\cdot^{\m A_r}\varphi_{qr}(y)$ for $x\in A_p, y\in A_q$ and $r=pq$. Moreover, if we define $\nein^{\m S} x = \nein^{\m A_p} x$ and $\no^{\m S} x = \no^{\m A_p} x$, for every $x\in A_p$ with $p$ positive, and
\[
x \le^{\m S} y \iff x\cdot^{\m S} \nein^{\m S} y = 0_{pq}, \quad \text{for all $x\in A_p$ and $y\in A_q$},
\]
then $\m S = \big(\biguplus A_p, \le^{\m S}, \cdot^{\m S}, \nein^{\m S}, \no^{\m S} \big)$ is $\m A$. Furthermore, cyclicity and commutativity are local properties.
\end{theorem}

\begin{proof}
By Remark~\ref{rem:Ap:partition:of:A}, the set $\{A_p : p\in A^+\}$ is a partition of $A$, and therefore $\biguplus A_p = A$. Given two elements $x\in A_p$ and $y\in A_q$, for arbitrary positive elements $p$ and $q$, and $r = pq$, we have that 
\[
x\cdot^{\m S} y = \varphi_{pr}(x)\cdot^{\m A_r} \varphi_{qr}(x) = rx\cdot ry = rrxy = rxy = 1_r\cdot(xy) = xy,
\]
since $r$ is positive, and therefore central and idempotent, and $xy \in A_r$ by Lemma~\ref{lem:prod:in:Aprod}. The involutive negations of every integral component $\m A_p$ are the restrictions of the corresponding operations of $\m A$, by Proposition~\ref{prop:integral:components}, and thus $\nein^{\m S} x = \nein^{\m A_p} x = \nein x$ and $\no^{\m S} x = \no^{\m A_p} x = \no x$.

Notice also that, by virtue of Lemma~\ref{lem:A:is:balanced}, we have that for $x\in A_p$ and $y\in A_q$, with $p$ and $q$ positive,
\[
x\le y \iff x\cdot\nein y = 0_{pq} \iff x\cdot^\m S\nein^\m S y = 0_{pq} \iff x\le^\m S y.   
\]
Finally, $\m A$ is commutative if and only if all its integral components are commutative, since commutativity is expressible by the regular equation $xy = yx$. The same is true for cyclicity because the involutive negations $\nein$ and $\no$ are unary operations defined componentwise.
\end{proof}

The previous theorem covers the structural decomposition results in~\cite{JiTuVa21} (see Theorem~\ref{thm:loc:integ:idempotent} and~Corollary~\ref{cor:idempotent:commutative} below). In that paper it is also shown that the variety of commutative idempotent involutive residuated lattices fails to be locally finite. Without the lattice operations, however, we have the following result. 

\begin{corollary}
Local finiteness is a local property of ipo\nbd-semigroups.
\end{corollary}

\begin{proof}
Suppose that the integral components of $\m A$ are locally finite and let $X\subset A$ be a finite set and $J=\{1_x : x\in X\}$. Without loss of generality, we can assume that $J$ is closed under binary joins (i.e., products), and that $J\subset X$. We will prove the proposition by induction on the cardinality of~$J$. Let $p$ be a minimal element in $J$ and $Y_p$ the closure of $X_p = X\cap A_p$ under products and involutive negations. Since $\m A_p$ is locally finite, $Y_p$ is also finite. Consider the finite set $X'=(X\setminus X_p)\cup\{ry : y\in Y_p,\ p < r \in J\}$ and notice that $J' = \{1_x : x\in X'\} = J\setminus\{p\}$, which is closed under binary joins, and $J'\subset X'$. By the inductive hypothesis, the subalgebra $\m B$ generated by $X'$ is finite. And since $J'$ is closed under binary joins, $B\subset \bigcup_{q\in J'} A_q$. Now, for any $y\in Y_p$ and $x\in B$, $yx =(ry)x\in B$ and $xy = x(ry)\in B$, where $r = p\cdot 1_x \in J\setminus\{p\}$. Since both $Y_p$ and $B$ are closed under products and involutive negations, the universe of the subalgebra generated by $X$ is $Y_p\cup B$, which is finite. The reciprocal is obvious.
\end{proof}

\begin{table}
\centering\small\tabcolsep3pt
\begin{tabular}{l|cccccccccc}
Number of elements $=$   &1&2&3&4&5& 6& 7& 8& 9\\\hline
ipo\nbd-semigroups   &1&4&10&48&160&933&4303\\
ipo\nbd-monoids      &1&3&5&20&39&179&500&2525\\\hline
locally integral ipo\nbd-semigroups &1&1&2&6&12&39&90&306\\
locally integral ipo\nbd-monoids &1&1&2&5&9&28&57&194&448\\
integral ipo\nbd-monoids &1&1&1&3&3&13&17&84&145
\end{tabular}
\caption{Number of ipo\nbd-semigroups up to isomorphism. Integral ipo\nbd-monoids are building blocks for locally integral ipo\nbd-semigroups.}
\label{tab:iposemigroups:size-n}
\end{table}

\section{Glueing Constructions}
\label{sec:glueing:constructions}

The last theorem of the previous section shows how every ipo\nbd-semigroup is an aggregate of its integral components. Our next question is, what are the conditions that a family of integral ipo\nbd-monoids and a compatible family of monoid homomorphisms between them should satisfy so that the construction of Theorem~\ref{thm:decomposition:ipo-monoids} is a (locally integral) ipo\nbd-semigroup?

To make this question precise, suppose that $\m D = (D,\lor)$ is a join-semilattice, $\cl A = \{\m A_p : p\in D\}$ is family of integral ipo\nbd-monoids, and $\Phi = \{\varphi_{pq}\: \m A_p\to \m A_q : p \le^{\m D} q\}$ is a compatible family of monoid homomorphisms. We call $(\m D, \cl A, \Phi)$ a \emph{semilattice directed system of integral ipo\nbd-monoids}.\footnote{~From a categorical point of view, a semilattice directed system of integral ipo\nbd-monoids can be understood as a functor $\Phi\:\ct D\to \IIPOMn$ from a skeletal and thin category $\ct D$ with binary coproducts to the category $\IIPOMn$ of integral ipo\nbd-monoids and monoid homomorphisms.\label{foot:categorical:semilat:direct:system}} Letting $\m A_p = (A_p, \le_p, \cdot_p, 1_p, \nein_p, \no_p)$, for all $p$ in $D$, we define the structure
\[
{\textstyle\int_\Phi \m A_p} = \big(\textstyle\biguplus\nolimits_D A_p, \le, \cdot, \nein, \no\,\big),
\]
where $\big(\biguplus_D A_p, \cdot \big)$ is the P\l{}onka sum of the family $\Phi$, and therefore a semigroup, and for all $p,q\in D$, $a\in A_p$, and $b\in A_q$, we define $\nein a = \nein_p a$ and $\no a = \no_p a$, and
\[
a\le b \iff a\cdot \nein b = 0_{p\lor q}.
\]
We call this structure $\int_\Phi \m A_p$ the \emph{glueing of $\cl A$ along the family $\Phi$}. With these definitions, one can see that given a locally integral ipo\nbd-semigroup $\m A$, its semilattice of positive elements $\m A^+$, its family $\cl A$ of integral components, and its associated family $\Phi=\{\varphi_{pq}\: \m A_p\to \m A_q\}$  of monoid homomorphisms determined by $\varphi_{pq}(x) = qx$, we have that $(\m A^+, \cl A, \Phi)$ is a semilattice directed system of integral ipo\nbd-monoids. We can restate Theorem~\ref{thm:decomposition:ipo-monoids} as saying that every locally integral ipo\nbd-semigroup $\m A$ is the glueing $\int_\Phi \m A_p$ of its integral components along the family of homomorphisms $\Phi$.\footnote{~That is to say, $\m A$ is completely described by the functor $\Phi^\m A\:\ct A^+\to\IIPOMn$ such that $\Phi^\m A(p)=\m A_p$ and $\Phi^\m A(p\le q) = \varphi_{pq}$ via the glueing construction. It is worth noticing that the decomposition of $\m A$ also induces a functor $F^\m A\:\ct A^+\to \ct{Poset}$, such that $F^\m A(p) = (A,\le)$ and $F^\m A(p\le q) = \varphi_{pq}$, since these maps are monotone. Hence, the Grothendieck construction of $F^\m A$ is a poset $\int F^\m A = (A,\le^*)$ determined as follows: for every $a\in A_p$ and $b\in A_q$, $a\le^* b$ if and only if $p\le q$ and $\varphi_{pq}(a)\le b$. We can easily see that ${\le^*} \subset {\le}$, but in general these are two distinct partial orders.}

\begin{proposition}\label{prop:glueing:is:an:aggregate}
If $\int_\Phi \m A_p$ is the glueing of a semilattice directed system of integral ipo\nbd-monoids $(\m D, \cl A, \Phi)$, then the restrictions of $\le$, $\cdot$, $\nein$, and $\no$ to $A_p$ are $\le_p$, $\cdot_p$, $\nein_p$, and $\no_p$, respectively. Moreover, for all $p \le^{\m D} q$ and $a\in A_p$, we have that $a\le \varphi_{pq}(a) = 1_q\cdot a$.
\end{proposition}

\begin{proof}
The fact that $\nein$ and $\no$ restricted to $A_p$ are $\nein_p$ and $\no_p$ is immediate, by the definitions. Now, if $a,b \in A_p$, then $p\lor p = p$, and by the definition of $\le$ we have that
${a \le b} \iff {a\cdot \nein b = 0_p}
\iff {\varphi_{pp}(a)\cdot_p\varphi_{pp}(\nein b) = 0_p}
\iff {a\cdot_p\nein b = 0_p}
\iff {a \le_p b}$,
since $\varphi_{pp}$ is the identity on $\m A_p$. For the same reason,
$a\cdot b = \varphi_{pp}(a)\cdot_p\varphi_{pp}(b) = a\cdot_p b$. Finally, if $p \le^{\m D} q$ and $a\in A_p$, then $a\cdot\nein\varphi_{pq}(a) = \varphi_{pq}(a)\cdot_q\nein\varphi_{pq}(a) = 0_q$. Hence $a \le \varphi_{pq}(a) = 1_q \cdot_q \varphi_{pq}(a) = \varphi_{qq}(1_q) \cdot_q \varphi_{pq}(a) = 1_q\cdot a$.
\end{proof}

\begin{remark}\label{rem:leG:reflexive}
An immediate consequence of this result is that $\le$ is a reflexive relation, since for every $p\in D$ and $a\in \m A_p$, we have that $a \le a$ if and only if $a \le_p a$, which we know is true. This result also implies that $\nein$ and $\no$ satisfy~(dn), since $\nein\no a = \nein_p\no_p a = a = \no_p\nein_p a = \no\nein a$. 
\end{remark}

The last proposition can be interpreted as saying that the glueing $\int_\Phi \m A_p$ is indeed an ``aggregate'' of the integral ipo\nbd-monoids $\m A_p$, although not necessarily an ipo\nbd-semigroup itself, since the relation $\le$ could fail to be antisymmetric and transitive. For instance, in the example of Figure~\ref{fig:counterexample:glueins:are:ipomonoids}, $0_r \le 0_p \le 1_p \le 1_q$, but $0_r \nle 1_q$. The natural question is then: given a semilattice directed system $(\m D, \cl A, \Phi)$ of integral ipo\nbd-monoids, what are the conditions that $\Phi$ must satisfy in order to ensure that $\int_\Phi \m A_p$ is an ipo\nbd-semigroup?

\begin{figure}
\centering\scriptsize
\begin{tikzpicture}[xscale=.8,yscale=.7]
\node at (-2,2)[n]{$\m D$};
\node at (9,3)[n]{$\m A_r$};
\node at (3,3)[n]{$\m A_p$};
\node at (6,3)[n]{$\m A_q$};
\node at (5,6.5)[n]{$A^+\cong D$};
\node at (5,-0.6)[n]{$A^\sim\cong D^\partial$};
\draw 
(9,6.5)--(3,5)(9,6.5)--(6,5)(9,-0.5)--(6,1)(9,-0.5)--(3,1)
(-1,3)node[label=right:$q$]{}--
(-2,4)node[label=above:$r$]{}--
(-3,3)node[label=left:$p$]{}
(9,-0.5)node[label=right:$0_r$]{}..controls(7.5,2)and(7.5,4)..
(9,6.5)node[label=right:$1_r$]{}..controls(10.5,4)and(10.5,2)..(9,-0.5)
(3,1)node[label=below:$0_p$]{}..controls(1.66,2.33)and(1.66,3.66)..(3,5)node[label=above:$1_p$]{}..controls(4.33,3.66)and(4.33,2.33)..(3,1)
(6,1)node[label=right:{\,\,\,$0_q$}]{}..controls(4.66,2.33)and(4.66,3.66)..(6,5)node[label=0:{\,\,\,\,\,$1_q$}]{}..controls(7.33,3.66)and(7.33,2.33)..(6,1);
\end{tikzpicture}
\qquad\qquad\qquad
\begin{tikzpicture}[xscale=.8,yscale=.7]
\node at (-4,3)[n]{$\m D$};
\node at (9,3)[n]{$\m A_r$};
\node at (3,3)[n]{$\m A_p$};
\node at (6,3)[n]{$\m A_q$};
\node at (.5,3)[n]{$\m A_1$};
\node at (5,6.5)[n]{$A^+\cong D$};
\node at (5,-0.6)[n]{$A^\sim\cong D^\partial$};
\draw 
(9,6.5)--(3,5)(9,6.5)--(6,5)(9,-0.5)--(6,1)(9,-0.5)--(3,1)(6,1)--(.5,2)--(3,1)(6,5)--(.5,4)--(3,5)
(-4,2)node[label=below:$1$]{}--
(-3,3)node[label=right:$q$]{}--
(-4,4)node[label=above:$r$]{}--
(-5,3)node[label=left:$p$]{}--(-4,2)
(9,-0.5)node[label=right:$0_r$]{}..controls(7.5,2)and(7.5,4)..
(9,6.5)node[label=right:$1_r$]{}..controls(10.5,4)and(10.5,2)..(9,-0.5)
(3,1)node[label=below:$0_p$]{}..controls(1.66,2.33)and(1.66,3.66)..(3,5)node[label=above:$1_p$]{}..controls(4.33,3.66)and(4.33,2.33)..(3,1)
(6,1)node[label=right:{\,\,\,$0_q$}]{}..controls(4.66,2.33)and(4.66,3.66)..(6,5)node[label=0:{\,\,\,\,\,$1_q$}]{}..controls(7.33,3.66)and(7.33,2.33)..(6,1)(.5,2)node[label=below:$0_1$]{}..controls(-.17,2.66)and(-.17,3.33)..(.5,4)node[label=above:$1_1$]{}..controls(1.16,3.33)and(1.16,2.66)..(.5,2);
\end{tikzpicture}
\caption{Structure of a locally integral ipo\nbd-semigroup and a locally integral ipo\nbd-monoid.}
\label{fig:structure:ipo-monoid}
\end{figure}
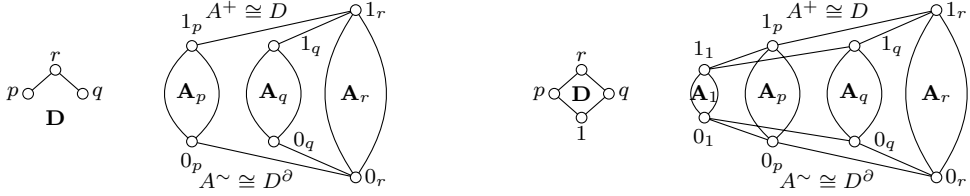

\begin{figure}
\centering\scriptsize
\begin{tikzpicture}[scale=2]
\node(p) at (-3.5,.5)[label=below:$p$]{};
\node(q) at (-4,1.5)[label=left:$q$]{};
\node(r) at (-3,1.5)[label=right:$r$]{};
\node(s) at (-3.5,2.5)[label=above:$s$]{};
\node(D)[n] at (-4.5,.5){\small$D$};

\draw (p)--(q)--(s)--(r)--(p);

\node(0p) at (0,0)[label=-30:$0_p$]{};
\node(0q) at (-1,1)[label=left:$0_q$]{};
\node(0r) at (1,1)[label=right:$0_r$]{};
\node(0s) at (0,2)[label=below:$0_s$]{};
\node(1p) at (0,1)[label=above:$1_p$]{};
\node(1q) at (-1,2)[label=35:$1_q$]{};
\node(1r) at (1,2)[label=145:$1_r$]{};
\node(1s) at (0,3)[label=above:$1_s$]{};
\node(P)[n] at (-1,0){\small$\Phi$};

\draw (0p)--(1p) (0q)--(1q) (0r)--(1r) (0s)--(1s);
\draw [-stealth, shorten >= 3pt, shorten <= 3pt] (1p) to [out=150,in=-30] (1q);
\draw [-stealth, shorten >= 3pt, shorten <= 3pt] (0p) to [out=150,in=-30] node[draw=none, fill=white, below]{\scriptsize$\varphi_{_{pq}}$} (1q);
\draw [-stealth, shorten >= 3pt, shorten <= 3pt] (1p) to [out=30,in=-150] (1r);
\draw [-stealth, shorten >= 3pt, shorten <= 3pt] (0p) to [out=30,in=-150] node[draw=none, fill=white, below]{\scriptsize$\varphi_{_{pr}}$} (1r);
\draw [-stealth, shorten >= 3pt, shorten <= 3pt] (1r) to [out=100,in=0] (1s);
\draw [-stealth, shorten >= 3pt, shorten <= 3pt] (0r) to [out=40,in=0] node[draw=none, fill=none, right]{\scriptsize\,\,$\varphi_{_{rs}}$} (1s);
\draw [-stealth, shorten >= 3pt, shorten <= 3pt] (1q) to [out=80,in=180] (1s);
\draw [-stealth, shorten >= 3pt, shorten <= 3pt] (0q) to [out=140,in=180] node[draw=none, fill=none, left]{\scriptsize$\varphi_{_{qs}}$\,} (1s);

\node(0p') at (3,1)[label=below:$0_p$]{};
\node(0q') at (4,0.5)[label=below:$0_q$]{};
\node(0r') at (5,0.5)[label=30:$0_r$]{};
\node(0s') at (6,0)[label=right:$0_s$]{};
\node(1p') at (3,2)[label=above:$1_p$]{};
\node(1q') at (4,2.5)[label=above:$1_q$]{};
\node(1r') at (5,2.5)[label=-30:$1_r$]{};
\node(1s') at (6,3)[label=right:$1_s$]{};

\node(G)[n] at (7,.5){\small$\int_\Phi \m A_p$};

\draw (0p')--(0q')--(0s')--(0r')--(0p')--(1p')--(1q')--(1s')--(1r')--(1p')
    (0q')--(1q') (0r')--(1r') (0s')--(1s');
\draw [dotted] (0r')--(1q') (0q')--(1r');

\node(nle)[n,fill=white] at (4.5,1.5){$\scriptstyle\nle$};
\end{tikzpicture}
\caption{A glueing of integral ipo\nbd-monoids that is not an ipo\nbd-monoid.}
\label{fig:counterexample:glueins:are:ipomonoids}
\end{figure}
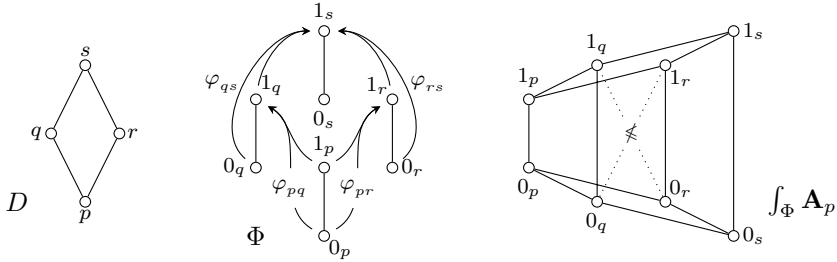

To answer this, our first step will be to examine the sets $G^+ = \{a \in \biguplus A_p : x \le a\cdot x, \text{for all } x\}$ and $G^\sim = \{\nein a : a \in G^+\}$ associated to any given glueing $\m G = \int_\Phi\m A_p$. As a corollary to Proposition~\ref{prop:glueing:is:an:aggregate}, we will show that the elements of $G^+$ are the elements of the form $1_p$, and the elements of $G^\sim$ are the ones of the form $0_p$, for some $p\in D$.

\begin{corollary}\label{cor:positive:and:zeros:in:G}
If $\m G=\int_\Phi \m A_p$ is the glueing of a semilattice directed system of integral ipo\nbd-monoids, then for all $p \in D$ and $a\in A_p$, we have that
\[
 a \in G^+ \iff a = 1_p \qquad\text{and}\qquad a \in G^\sim \iff a = 0_p.
\]
\end{corollary}

\begin{proof}
If $a \in G^+$ then, in particular, $1_p \le a \cdot 1_p = a \cdot_p 1_p = a \le 1_p$, by Proposition~\ref{prop:glueing:is:an:aggregate}, and therefore $a = 1_p$. Conversely, suppose that $p,q\in D$, $x \in A_q$, and $r=p\lor q$. Also by Proposition~\ref{prop:glueing:is:an:aggregate}, we have that $x\le 1_r\cdot x = \varphi_{pr}(1_p)\cdot_r\varphi_{qr}(x) = 1_p\cdot x$, and hence $1_p \in G^+$. For the second part, note that $a \in G^\sim \iff a = \nein 1_p = 0_p$ by the first part.
\end{proof}

Reflecting on Proposition~\ref{prop:A+:join:semilattice}, we would like to show that the relation $\le$ endows $G^+$ with the structure of a join-semilattice isomorphic to $\m D$, and $G^\sim$ with the structure of a meet-semilattice dually isomorphic to~$\m D$. In general, this will not be true. For this to hold, it will be necessary to assume an extra property of $\Phi$. We will prove first that this property is valid for the family of monoid homomorphisms associated to a locally integral ipo\nbd-semigroup.

\begin{lemma}\label{lem:A:liipomon:condition:on:zeros}
Let $\m A$ be a locally integral ipo\nbd-semigroup and $p\le q$ positive elements of $\m A$. Then, $\varphi_{pq}(0_p) = 0_q$ if and only if $p = q$.
\end{lemma}

\begin{proof}
The implication from right to left is obvious, since $p = q$ implies that $\varphi_{pq}$ is the identity map. As for the other implication, just notice that $p < q$ implies that $0_q < 0_p \le q\cdot 0_p = \varphi_{pq}(0_p)$, and therefore $\varphi_{pq}(0_p) \neq 0_q$.
\end{proof}

This suggests the following condition for $\Phi$, which we call \emph{zero avoidance}:
\begin{equation}\label{eq:condition:on:zeros}
\text{for all } p \le^{\m D} q,\quad \varphi_{pq}(0_p) = 0_q \iff p = q. \tag{za}    
\end{equation}

\begin{lemma}\label{lem:za:implies:iso:D:positives}
If $\int_\Phi \m A_p$ is the glueing of a semilattice directed system of integral ipo\nbd-monoids $(\m D, \cl A, \Phi)$ and $\Phi$ satisfies~\eqref{eq:condition:on:zeros}, then for all $p$ and $q$ in $D$,
$1_p \le 1_q \iff p \le^{\m D} q$ and
$0_p \le 0_q \iff q \le^{\m D} p$.
\end{lemma}

\begin{proof}
For all $p,q\in D$, with $r = p\lor q$, we have the following equivalences:
\begin{align*}
1_p \le 1_q &\iff \varphi_{pr}(1_p)\cdot_r\varphi_{qr}(0_q) = 0_r 
    \iff 1_r\cdot_r\varphi_{qr}(0_q) = 0_r
    \iff \varphi_{qr}(0_q) = 0_r \\
    &\iff q = r = p\lor q
    \iff p \le^{\m D} q.
\end{align*}
As for the second equivalence, $0_p \le 0_q \iff \nein 1_p \le \nein 1_q \iff 1_q\le 1_p \iff q\le^\m D p$.
\end{proof}

Next, recall that by virtue of Lemma~\ref{lem:A:is:balanced}, in any locally integral ipo\nbd-semigroup $\m A$, for all positive elements $p$ and $q$ and $x\in A_p$ and $y\in A_q$, we have that $x\le y \iff x\cdot\nein y = 0_{pq} \iff \no y\cdot x = 0_{pq}$. Therefore, for a glueing $\int_\Phi \m A_p$ to be a locally integral ipo\nbd-semigroup, the family $\Phi$ has to satisfy the corresponding condition:
\begin{equation}\label{eq:cond:equivalent:to:Phi:satisfies:bal}
    \text{for all } p,q\in D, a\in A_p, b\in A_q,\quad
    a\cdot\nein b = 0_{p\lor q} \iff \no b\cdot a = 0_{p\lor q}.\tag{*}
\end{equation}

As we will see in the next lemma, this is equivalent to the requirement that all the morphisms in $\Phi$ satisfy the following \emph{balance} property:
\begin{equation}\label{eq:Phi:satisfies:bal}
\text{for every } a\in A_p,\quad 
\nein\varphi_{pq}(\no a) = \no\varphi_{pq}(\nein a).\tag{bal}
\end{equation}

\begin{lemma}
If $\int_\Phi \m A_p$ is the glueing of a semilattice directed system of integral ipo\nbd-monoids $(\m D, \cl A, \Phi)$, then $\int_\Phi\m A_p$ satisfies~\eqref{eq:cond:equivalent:to:Phi:satisfies:bal} if and only if every morphism in $\Phi$ satisfies~\eqref{eq:Phi:satisfies:bal}.
\end{lemma}

\begin{proof}
Suppose that $\Phi$ satisfies~\eqref{eq:Phi:satisfies:bal}, and let $p,q\in D$, $a\in A_p$ and $b\in A_q$, and $r = p\lor q$.
\begin{align*}
a\cdot\nein b = 0_r &\iff \varphi_{pr}(a)\cdot_r\varphi_{qr}(\nein b) = 0_r
\iff \varphi_{pr}(a)\le_r\no\varphi_{qr}(\nein b)\\
&\iff \varphi_{pr}(a)\le_r\nein\varphi_{qr}(\no b)
\iff \varphi_{qr}(\no b)\cdot_r\varphi_{pr}(a) = 0_r
\iff \no b\cdot a =0_r.
\end{align*}
For the other direction, suppose that $\int_\Phi\m A_p$ satisfies~\eqref{eq:cond:equivalent:to:Phi:satisfies:bal} and let $p\le q$ in $D$ and $a\in A_p$. The property~\eqref{eq:Phi:satisfies:bal} follows from the fact that for every $x\in A_q$, we have that
\begin{align*}
x\le_q\no\varphi_{pq}(\nein a) &\iff x\cdot_q\varphi_{pq}(\nein a) = 0_q
\iff \varphi_{qq}(x)\cdot_q\varphi_{pq}(\nein a) = 0_q\\
&\iff x\cdot\nein a = 0_q
\iff \no a\cdot x = 0_q
\iff \varphi_{pq}(\no a)\cdot_q\varphi_{qq}(x) = 0_q\\
&\iff \varphi_{pq}(\no a)\cdot_q x = 0_q
\iff x\le_q \nein\varphi_{pq}(\no a)\tag*{\QED}
\end{align*}\def\QED{}
\end{proof}

We say that $\Phi$ is \emph{balanced} if all the morphisms in $\Phi$ are balanced. One can readily check that the commutativity of $\int_\Phi \m A_p$ implies that $\Phi$ is balanced. We will prove next that when $\Phi$ is balanced, the operations $\nein$ and $\no$ are involutive with respect to the relation~$\le$ and the glueing satisfies rotation.

\begin{lemma}\label{lem:glueing:involution}
If $\int_\Phi \m A_p$ is the glueing of a semilattice directed system of integral ipo\nbd-monoids $(\m D, \cl A, \Phi)$ such that $\Phi$ is balanced, then for all $p,q\in D$, $a\in A_p$, and $b\in A_q$,
\[
a \le b \iff \no b \le \no a \iff \nein b \le \nein a.
\]
Moreover, for all $p,q,r\in D$, $a\in A_p$, $b\in A_q$, and $c\in A_r$,
\[
a\cdot b\le c \iff b\cdot\nein c \le \nein a \iff \no c\cdot a\le \no b.
\]
\end{lemma}

\begin{proof}
The first equivalence can be proven as follows:
$a \le b \iff a \cdot \nein b = 0_{p\lor q}
    \iff {\no b\cdot a = 0_{p\lor q}}
    \iff \no b\cdot \nein\no a = 0_{p\lor q}
    \iff \no b \le \no a$.
As for the second one,
$\nein b \le \nein a \iff a = \no\nein a \le \no\nein b = b$.

Regarding rotation, suppose that $a \in A_p$, $b \in A_q$, $c \in A_r$ and $u=p\lor q\lor r$. Then ${a\cdot b \le c} \iff {(a\cdot b)\cdot \nein c = 0_u} \iff \no\nein a \cdot (b\cdot \nein c) = 0_u \iff (b \cdot \nein c)\cdot \nein\nein a = 0_u$ by~(bal), that is, {$b\cdot \nein c \le \nein a$}. For the second equivalence, $a\cdot b \le c \iff {(a \cdot b) \cdot \nein c = 0_u} \iff \no {c \cdot (a\cdot b) = 0_u} \iff {(\no c \cdot a) \cdot b = 0_u} \iff (\no c \cdot a) \cdot \nein\no b = 0_u \iff \no c \cdot a \le \no b$. 
\end{proof}

The semilattice directed system of Figure~\ref{fig:counterexample:glueins:are:ipomonoids} avoids zeros and is balanced\,---since its components are commutative, and therefore cyclic---, but the glueing is not an ipo\nbd-semigroup because the relation $\le$ is not transitive. Hence, we also need to impose the following condition on the family $\Phi$ for $\int_\Phi \m A_p$ to be an ipo\nbd-semigroup:
\begin{equation}\label{eq:transitivity:leG}
\text{for all } a,b,c\in \textstyle\biguplus A_p,\quad \text{if } a \le b \text{ and } b \le c, \text{ then } a \le c. \tag{tr}    
\end{equation}
Our next result characterizes the condition~\eqref{eq:transitivity:leG} in simpler terms.

\begin{lemma}\label{lem:char:transitivity:leG}
If $\int_\Phi \m A_p$ is the glueing of a semilattice directed system of integral ipo\nbd-monoids $(\m D, \cl A, \Phi)$ and $\Phi$ satisfies~\eqref{eq:condition:on:zeros} and~\eqref{eq:Phi:satisfies:bal}, then $\Phi$ satisfies~\eqref{eq:transitivity:leG} if and only if it satisfies:
\begin{enumerate}
    \item for all $p \le^{\m D} q$, and $a,b\in A_p$, 
    $a \le_p b \implies \varphi_{pq}(a) \le_q \varphi_{pq}(b)$;\hfill \textnormal{(mon)}
    
    \item for all $p \le^{\m D} q$, $p \le^{\m D} r$, and $a\in A_p$,
    $\nein \varphi_{pq}(a) \le \varphi_{pr}(\nein a)$;\hfill \textnormal{(lax)}
\end{enumerate}
\end{lemma}

\begin{shortproof}
Suppose that $\Phi$ satisfies both~\eqref{eq:Phi:satisfies:bal} and~\eqref{eq:transitivity:leG}.
\begin{enumerate}
    \item By Proposition~\ref{prop:glueing:is:an:aggregate},  $a \le b \le \varphi_{pq}(b)$, and by~\eqref{eq:transitivity:leG}, we obtain that $a \le \varphi_{pq}(b)$. Hence, $\varphi_{pq}(a)\cdot_q \nein \varphi_{pq}(b) = a \cdot \nein \varphi_{pq}(b) = 0_q$, and therefore $\varphi_{pq}(a) \le_q \varphi_{pq}(b)$.

    \item By Proposition~\ref{prop:glueing:is:an:aggregate}, we have that $\nein a \le \varphi_{pr}(\nein a)$ and $a \le \varphi_{pq}(a)$, and then by Lemma~\ref{lem:glueing:involution}, $\nein \varphi_{pq}(a) \le \nein a$. We deduce by~\eqref{eq:transitivity:leG} that $\nein\varphi_{pq}(a) \le \varphi_{pr}(\nein a)$.
\end{enumerate}

In order to prove the reverse implication, suppose that $\Phi$ satisfies \eqref{eq:condition:on:zeros}, \eqref{eq:Phi:satisfies:bal}, (mon), and (lax), and $p,q,r\in D$, with $s = p\lor q$, $t = q\lor r$, $u = p\lor r$, $a\in A_p$, $b\in A_q$, and $c\in A_r$ are such that $a\le b$ and $b \le c$. Then, by definition of~$\le$, we have that
$\varphi_{ps}(a)\cdot_s\varphi_{qs}(\nein b) = 0_s$ and
$\varphi_{qt}(b)\cdot_t\varphi_{rt}(\nein c) = 0_t$,
whence we deduce that
$\varphi_{ps}(a) \le_s \no\varphi_{qs}(\nein b)$ and
$\varphi_{rt}(\nein c) \le_t \nein\varphi_{qt}(b)$.
Taking $v = s\lor t$, we deduce by~(mon) that
$\varphi_{pv}(a) \le_v \varphi_{sv}(\no\varphi_{qs}(\nein b))$ and
$\varphi_{rv}(\nein c) \le_v \varphi_{tv}(\nein\varphi_{qt}(b))$.
Moreover, by~(lax), we have that
$\nein\varphi_{qt}(b) \le \varphi_{qs}(\nein b)$
and by Lemma~\ref{lem:glueing:involution}, we deduce that
$\no\varphi_{qs}(\nein b) \le \no\nein\varphi_{qt}(b) = \varphi_{qt}(b)$,
and therefore
\[
\varphi_{pv}(a)\cdot_v\varphi_{rv}(\nein c) \le_v \varphi_{sv}(\no\varphi_{qs}(\nein b)) \cdot_v \varphi_{tv}(\nein\varphi_{qt}(b)) = 0_v,
\]
so $\varphi_{pv}(a)\cdot_v\varphi_{rv}(\nein c) = 0_v$ and hence
$\varphi_{uv}(\varphi_{pu}(a)\cdot_u\varphi_{ru}(\nein c)) = \varphi_{uv}(\varphi_{pu}(a))\cdot_v\varphi_{uv}(\varphi_{ru}(\nein c)) = \varphi_{pv}(a)\cdot_v\varphi_{rv}(\nein c) = 0_v$. By~\eqref{eq:condition:on:zeros}, $v=u=p\lor r$ and $\varphi_{pu}(a)\cdot_u\varphi_{ru}(\nein c) = 0_v$, that is,
$a\le c$.
\end{shortproof}

\begin{remark}
Notice that if a compatible family $\Phi$ satisfies~\eqref{eq:Phi:satisfies:bal}, then it satisfies~(lax) if and only if for all $p \le^{\m D} q$, $p \le^{\m D} r$, and $a\in A_p$, $\no \varphi_{pq}(a) \le \varphi_{pr}(\no a)$.
\end{remark}

\begin{lemma}\label{lem:antisymmetry:leG}
If $\int_\Phi \m A_p$ is the glueing of a semilattice directed system of integral ipo\nbd-monoids $(\m D, \cl A, \Phi)$ and $\Phi$ satisfies~\eqref{eq:condition:on:zeros}, \eqref{eq:Phi:satisfies:bal}, and~\textnormal{(lax)}, then $\le$ is antisymmetric.
\end{lemma}

\begin{proof}
Suppose that $p,q\in D$ with $r = p\lor q$, and $a\in A_p$ and $b\in A_q$ are such that $a \le b$ and $b \le a$. That is, $\varphi_{pr}(a)\cdot_r \varphi_{qr}(\nein b) = 0_r$ and $\varphi_{qr}(b)\cdot_r \varphi_{pr}(\nein a) = 0_r$, or equivalently $\varphi_{pr}(a) \le_r \no\varphi_{qr}(\nein b)$ and $\varphi_{qr}(b) \le_r \no\varphi_{pr}(\nein a)$. By~(lax) and the preceding remark, we get
\[
\varphi_{pr}(a) \le_r \no\varphi_{qr}(\nein b)
\le_r \varphi_{qr}(\no\nein b) =
\varphi_{qr}(b) \le_r \no\varphi_{pr}(\nein a).
\]
Hence, we would have that
$\varphi_{pr}(0_p) = \varphi_{pr}(a\cdot_p \nein_pa) = \varphi_{pr}(a)\cdot_r \varphi_{pr}(\nein_pa) = 0_r$.
By~\eqref{eq:condition:on:zeros}, this only is possible if $p = r$. By a symmetric argument, we also obtain that $q = r$, and therefore $p = q$. Thus, by Proposition~\ref{prop:glueing:is:an:aggregate}, we have that $a \le_p b$ and $b \le_p a$, and therefore $a = b$.
\end{proof}

We are now in the position to prove our main result.

\begin{theorem}\label{thm:decompsition:and:glueing}
A structure $\m A$ is a locally integral ipo\nbd-semigroup if and only if there is a semilattice directed system $(\m D,\cl A, \Phi)$ of integral ipo\nbd-monoids satisfying~\eqref{eq:condition:on:zeros}, \eqref{eq:Phi:satisfies:bal},
(mon), and~(lax) such that $\m A = \int_\Phi \m A_p$. Moreover, $\m A$ has a global identity if and only if $\m D$ has a minimum element.
\end{theorem}

\begin{proof}
As we showed in Theorem~\ref{thm:decomposition:ipo-monoids}, if $\m A$ is a locally integral ipo\nbd-semigroup, then $\m A^+$ is a join-semilattice, the integral components of $\m A$ form a family $\{\m A_p : p\in A^+\}$ of integral ipo\nbd-monoids, and $\Phi = \{\varphi_{pq}\: \m A_p \to \m A_q : p \le q\}$, where $\varphi_{pq}(x) = qx$, is a compatible family of monoid homomorphisms, such that  $\m A = \int_\Phi \m A_p$. Moreover, $\Phi$ is balanced by Lemma~\ref{lem:A:is:balanced}, avoids zeros by Lemma~\ref{lem:A:liipomon:condition:on:zeros}, and satisfies condition~\eqref{eq:transitivity:leG}, and therefore~(mon) and~(lax), since $\le$ is a partial order.

Conversely, if $\m D = (D,\lor)$ is a join-semilattice, $\{\m A_p : p\in D\}$ is a family of integral ipo\nbd-monoids, and $\Phi = \{\varphi_{pq}\: \m A_p \to \m A_q : p \le^{\m D} q\}$ is a compatible family of monoid homomorphisms satisfying~\eqref{eq:Phi:satisfies:bal}, \eqref{eq:condition:on:zeros}, (mon), and~(lax), then $\le$ is a reflexive binary relation on $A=\biguplus A_p$ by Remark~\ref{rem:leG:reflexive}, which is also transitive by Lemma~\ref{lem:char:transitivity:leG}, and antisymmetric by Lemma~\ref{lem:antisymmetry:leG}. That is, $\big(\biguplus A_p, {\le}\,\big)$ is a poset. By construction as a P\l{}onka sum of monoids, $\big(\biguplus A_p, \cdot\big)$ is a semigroup. Furthermore, the involutive negations $\no$ and $\nein$ satisfy~(dn), by Remark~\ref{rem:leG:reflexive}, and since $\Phi$ satisfies~\eqref{eq:Phi:satisfies:bal} they are antitone by Lemma~\ref{lem:glueing:involution} and $\int_\Phi \m A_p$ satisfies~(rot). Hence, $\int_\Phi \m A_p$ is an ipo\nbd-semigroup.

It can be readily checked that for all $p\in D$ and $x\in A_p$, $\no x\cdot x = x\cdot\nein x$, since these involutive negations and products are computed inside $\m A_p$, which is integral. Moreover, $1_x = \no (x \cdot \nein x) = \no_p (x \cdot_p \nein_p x) = 1_p$, which by Corollary~\ref{cor:positive:and:zeros:in:G} is positive, and $1_x\cdot x = 1_p\cdot x = x$. For the same reasons, one can check that $x \le 1_p = 1_x$ and $x\backslash 1_x = x\backslash 1_p = \nein(0_p\cdot x) = \nein 0_p = 1_p = 1_x$, since the restriction of $\le$ to $A_p$ is $\le_p$ and the product and negation are computed inside $\m A_p$, by Proposition~\ref{prop:glueing:is:an:aggregate}, and $\m A_p$ is integral. In summary, $\int_\Phi \m A_p$ is a locally integral ipo\nbd-semigroup. 

Since for all $p\in D$ and $x\in A_p$, $1_x = 1_p$, we deduce that $\{\m A_p : p\in D\}$ is the family of integral components of $\int_\Phi \m A_p$. Also, by Proposition~\ref{prop:glueing:is:an:aggregate}, we know that $\varphi_{pq}(x) = 1_q\cdot x$, for all $p \le^{\m D} q$ and $x\in A_p$, that is, $\Phi$ is the family of homomorphisms of the decomposition of Theorem~\ref{thm:decomposition:ipo-monoids}.

Finally, if the locally integral ipo\nbd-semigroup $\m A$ has a global identity $1$, then it is the smallest of all the identities of $\m A$, that is, the minimum of $\m A^+$, as we already observed at the beginning of Section~\ref{sec:locally:integral:ipo-semigroups}. For the other direction, by Lemma~\ref{lem:za:implies:iso:D:positives}, there is an isomorphism between the semilattices $\m A^+$ and~$\m D$. Therefore, if $\m D$ has a minimum element $p$, so does $\m A^+$, namely, $1_p$. And, as we mentioned before Proposition~\ref{prop:A+:join:semilattice}, this implies that $1_p$ is the global identity of $\m A$.
\end{proof}

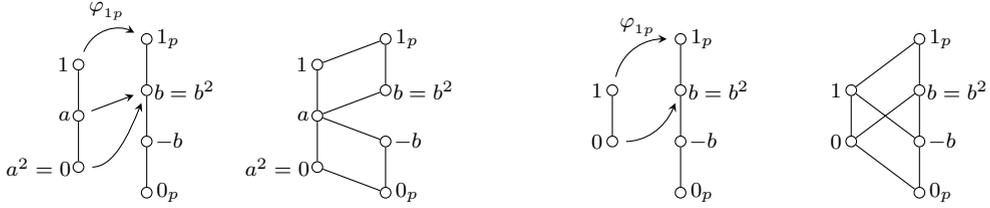
\begin{figure}[t]
\centering\scriptsize
\begin{tikzpicture}[yscale=1.5]
\node(1) at (0,1)[label=left:$1$]{};
\node(a) at (0,0)[label=left:$a$]{};
\node(0) at (0,-1)[label=left:{$a^2=0$}]{};
\node(1p) at (2,1.5)[label=right:$1_p$]{};
\node(b) at (2,0.5)[label=right:{$b=b^2$}]{};
\node(-b) at (2,-0.5)[label=right:$\no b$]{};
\node(0p) at (2,-1.5)[label=right:$0_p$]{};

\draw (0)--(a)--(1)
    (0p)--(-b)--(b)--(1p);
\draw [-stealth, shorten >= 3pt, shorten <= 3pt] (1) to [out=60,in=160] node[draw=none, fill=none, above]{\scriptsize$\varphi_{_{1p}}$} (1p);
\draw [-stealth, shorten >= 3pt, shorten <= 3pt] (a) to (b);
\draw [-stealth, shorten >= 3pt, shorten <= 3pt] (0) to [out=0,in=-120] (b);

\node(1') at (7,1)[label=left:$1$]{};
\node(a') at (7,0)[label=left:$a$]{};
\node(0') at (7,-1)[label=left:{$a^2=0$}]{};
\node(1p') at (9,1.5)[label=right:$1_p$]{};
\node(b') at (9,0.5)[label=right:{$b=b^2$}]{};
\node(-b') at (9,-0.5)[label=right:$\no b$]{};
\node(0p') at (9,-1.5)[label=right:$0_p$]{};

\draw (0')--(a')--(1')--(1p')--(b')--(a')--(-b')--(0p')--(0');
\end{tikzpicture}
\qquad\qquad\qquad
\begin{tikzpicture}[yscale=1.5]
\node(1) at (0,.5)[label=left:$1$]{};
\node(0) at (0,-.5)[label=left:{$0$}]{};
\node(1p) at (2,1.5)[label=right:$1_p$]{};
\node(b) at (2,0.5)[label=right:{$b=b^2$}]{};
\node(-b) at (2,-0.5)[label=right:$\no b$]{};
\node(0p) at (2,-1.5)[label=right:$0_p$]{};

\draw (0)--(1)
    (0p)--(-b)--(b)--(1p);
\draw [-stealth, shorten >= 3pt, shorten <= 3pt] (1) to [out=70,in=180] node[draw=none, fill=none, above]{\scriptsize$\varphi_{_{1p}}$} (1p);
\draw [-stealth, shorten >= 3pt, shorten <= 3pt] (0) to [out=0,in=-120] (b);

\node(1') at (7,.5)[label=left:$1$]{};
\node(0') at (7,-.5)[label=left:{$0$}]{};
\node(1p') at (9,1.5)[label=right:$1_p$]{};
\node(b') at (9,0.5)[label=right:{$b=b^2$}]{};
\node(-b') at (9,-0.5)[label=right:$\no b$]{};
\node(0p') at (9,-1.5)[label=right:$0_p$]{};

\draw (0')--(1')--(1p')--(b')--(0')--(0p')--(-b')--(1') (b')--(-b');
\end{tikzpicture}
\caption{Two glueings, one being a semiring, the other just an ipo\nbd-monoid.}
\label{fig:two:glueings}
\end{figure}

\begin{corollary}
Given any nonempty family of nontrivial integral ipo\nbd-monoids (involutive semirings, respectively) there is a locally integral ipo\nbd-semigroup (involutive semiring, respectively) whose integral components are the given ones. Furthermore, we can construct it so that it has a global identity.
\end{corollary}

\begin{proof}
If $\{\m A_p : p\in D\}$ is a nonempty family of nontrivial ipo\nbd-monoids, let's choose a linear order on $D$ and let $\m D=(D,\lor)$ be the associated join-semilattice. Let $\Phi = \{\varphi_{pq}\: \m A_p \to \m A_q : p \le^{\m D} q\}$ be the family of maps such that $\varphi_{pp}$ is the identity map and $\varphi_{pq}(x) = 1_q$ if $p < q$. Then $\Phi$ is a compatible family of monoid homomorphisms satisfying~\eqref{eq:condition:on:zeros}, \eqref{eq:Phi:satisfies:bal}, (mon), and~(lax). By Theorem~\ref{thm:decompsition:and:glueing}, $\int_\Phi \m A_p$ is a locally integral ipo\nbd-monoid whose integral components are $\{\m A_p : p\in D\}$. Moreover, for every $a\in A_p$, $b\in A_q$ with $p< q$, $a\le b$ if and only if $b=1_q$, and $b\le q$ if and only if $a=0_p$.

If in addition all the integral components are involutive semirings, then $\int_\Phi \m A_p$ is also an involutive semiring, since the join of $a\in A_p$ and $b\in A_q$ in $\int_\Phi \m A_p$ is their join in $\m A_p$, if $p = q$, or $1_q$ if $p<q$ and $b\neq 0_q$, or $a$ if $p < q$ and $b = 0_q$. If moreover we choose $(D,\lor)$ to have a lower bound $p$, then  $1_p$ is the global identity of $\int_\Phi\m A_p$.
\end{proof}

The last theorem is the object correspondence of a categorical equivalence. First, for any arbitrary category $\ct C$, we can define the category $\ct C^\SD$ of semilattice directed systems on $\ct C$ as the category whose objects are functors $\Phi\:\ct D\to \ct C$, where $\ct D$ is a skeletal and thin category with binary coproducts (i.e., a join-semilattice). A morphism $\Phi\to\Psi$ in $\ct C^\SD$ is a pair $\pair{\tau,\eta}$, where ${\tau\:\ct D_\Phi\to\ct D_\Psi}$ is a coproduct-preserving functor and $\eta\: \Phi\nat\Psi\tau$ is a natural transformation. 

Now, let $\LIIPOSg$ be the category whose objects are the locally integral ipo\nbd-semigroups and morphisms are the monotone algebraic homomorphisms. Let $\IIPOMn$ be the category of integral ipo\nbd-monoids and monoid homomorphisms, and let $\SDIIPOMn$ be the full subcategory of $\IIPOMn^\SD$ whose objects are the functors $\Phi\:\ct{D}\to \IIPOMn$ satisfying~\eqref{eq:condition:on:zeros}, \eqref{eq:Phi:satisfies:bal}, (mon), and~(lax).
\footnote{~Notice that both (mon) and (bal) could also be encoded as conditions on the morphisms of $\IIPOMn$, but (za) and (lax) cannot.}

Every morphism $\eta\:\m A\to\m B$ induces a semilattice homomorphism $\tau\: \m A^+\to\m B^+$ and monoid homomorphisms $\eta_p\:\m A_p\to \m B_{\tau(p)}$\,---which are the corresponding restrictions of~$\eta$---\,such that 
\[
\text{for all $p\le q$ in $A^+$,}\quad \eta_q\varphi_{pq}^\m A = \varphi_{\tau(p)\tau(q)}^\m B\eta_p.\tag{nat}
\]
Conversely, any semilattice homomorphism $\tau\:\m A^+\to\m B^+$ and family of monoid homomorphisms $\eta_p\:\m A_p\to\m B_{\tau(p)}$ satisfying~(nat) induce a morphism $\eta\:\m A\to\m B$. That is, every morphism $\eta\:\m A\to \m B$ in $\LIIPOSg$ corresponds to a morphism $\pair{\tau,\eta}\:\Phi^\m A\to\Phi^\m B$ in $\SDIIPOMn$. For the other direction, every object $\Phi$ of $\SDIIPOMn$ gives rise to an object $\int\Phi$ in $\LIIPOSg$, and every morphism ${\pair{\tau,\eta}\:\Phi\to\Psi}$ induces a morphism $\eta\:\int\Phi\to\int\Psi$ in $\LIIPOSg$. This correspondence is an equivalence of categories between $\LIIPOSg$ and $\SDIIPOMn$.

A morphism in $\LIIPOSg$ is an \emph{embedding} if it is order-reflecting, and therefore injective. $\m A$ is a \emph{substructure} of $\m B$ if and only if the inclusion map $A\to B$ is an embedding, that is, if $\m A^+$ is a subsemilattice of $\m B^+$, for every $p\in A^+$, $\m A_p$ is a substructure of $\m B_p$, and for all $p\le q$ in $A^+$, the associated monoid homomorphisms $\varphi_{pq}^\m A\:\m A_p\to\m A_q$ is the restriction of the corresponding $\varphi_{pq}^\m B\:\m B_p\to \m B_q$. Also, $\m A$ is isomorphic to $\m B$ if and only if there is a semilattice isomorphism $\tau\:\m A^+\to \m B^+$ and isomorphisms $\eta_p\:\m A_p\to \m B_{\tau(p)}$ for all $p\in A^+$ satisfying~(nat).

We have already characterized which locally integral ipo\nbd-semigroups are the $1$-free reducts of the locally integral ipo\nbd-monoids, namely those that have a global identity, or equivalently, those that have a smallest positive element. We end this section by responding to a very natural question: which locally integral ipo\nbd-semigroups are the subreducts of locally integral ipo\nbd-monoids? Notice that, by the discussion above, checking if a locally integral ipo\nbd-semigroup is a substructure of or isomorphic to another one can be done componentwise.

\eject

\begin{theorem}
For any locally integral ipo\nbd-semigroup $\m A$, the following conditions are equivalent:
\begin{enumerate}
    \item $\m A$ is a subreduct of a locally integral ipo\nbd-monoid.
    \item $\m A$ satisfies the equation $0_x0_y = 0_{xy}$.
    \item For all $p,q\in A^+$, $0_p0_q = 0_{pq}$.
    \item For all $p,q\in A^+$, $0_p \le 1_q$.
\end{enumerate}
Moreover, the same conditions characterize the locally integral ipo\nbd-semirings that are subreducts of a locally integral unital semiring.
\end{theorem}

\begin{proof}
$(1\implies 2)$ If $\m B$ is a locally integral ipo\nbd-monoid, then for all $x,y\in B$ we have that $1_x\cdot 0_x = 0_x\le 0 \le 1 \le 1_y$, and by rotation we obtain that $0_x\cdot\nein 1_y\le \nein 1_x$, that is, $0_x0_y\le 0_x$. Analogously, we can prove that $0_x0_y\le 0_y$, and therefore $0_x0_y\le 0_x\land 0_y = 0_{xy} = 0_{xy}0_{xy}\le 0_x0_y$. That is, $0_x0_y = 0_{xy}$ is valid in $\m B$. Therefore, if $\m A$ is a subreduct of $\m B$, the equation is also valid in $\m A$.

\vspace\itemsep
\noindent
$(2\implies 3)$ This implication is trivial.

\vspace\itemsep
\noindent
$(3\implies 4)$ Since $0_{pq} = 0_p\land 0_q \le 0_p$, the equation $0_p0_q\le 0_{pq}$ implies that $0_p0_q\le 0_p$, and by rotation we have that $\no 0_p\cdot 0_p \le \no 0_q$, that is, $0_p\le 1_q$.

\vspace\itemsep
\noindent
$(4\implies 1)$ We will distinguish two cases according to whether one of the integral components of $\m A$ is trivial or not:

If one of the integral components of $\m A$ is trivial, that is, there is $p\in A^+$ such that $0_p = 1_p$, we can readily check that $p=1_p$ is the minimum of $\m A^+$, since for every $q\in A^+$ we have that $1_p = 0_p \le 1_q = q$. Thus, $1_p$ is the global identity of $\m A$ and the structure $\widehat{\m A} = (A,\le,\cdot,1_p,\nein,\no)$ is a locally integral ipo\nbd-monoid, as a consequence of Proposition~\ref{prop:char:locally:integral}.

If for every $p\in A^+$, $0_p < 1_p$, let $(\m A^+,\cl A,\Phi)$ be the semilattice directed system associated to~$\m A$. Consider the join-semilattice $\m D = (A^+\cup\{\bot\}, \lor)$ that results from adding a lower bound to $\m A^+$ and let $\m A_\bot$ be any integral ipo\nbd-monoid and $\widehat{\cl A} = \cl A\cup\{\m A_\bot\}$. Let $\varphi_{\bot\bot}\: \m A_\bot\to\m A_\bot$ be the identity homomorphism, $\varphi_{\bot p}\:\m A_\bot\to\m A_p$ the monoid homomorphism determined by $\varphi_{\bot p}(a) = 1_p$, and $\widehat\Phi = \Phi\cup\{\varphi_{\bot p} : p\in D\}$. One can check that $(\m D,\widehat{\cl A},\widehat\Phi)$ is a semilattice directed system of integral ipo\nbd-monoids, and we can form its glueing $\m B = \int_\Phi\m A_p$. We will see next that $\widehat\Phi$ satisfies all the necessary conditions for $\m B$ to be a locally integral ipo\nbd-semigroup, but first of all notice that, by construction, $\m A$ is a substructure of $\m B$.

The family $\widehat\Phi$ avoids zeros, because $\Phi$ does and, moreover, for every $p\in A^+$, $\varphi_{\bot p}(a) = 1_p\neq 0_p$. In order to show that it satisfies~\eqref{eq:Phi:satisfies:bal}, 
we only need to check that the newly added morphisms satisfy~\eqref{eq:Phi:satisfies:bal}: the identity $\varphi_{\bot\bot}$ is trivially balanced, and for every $p\in D$ and $a\in A_\bot$, we have that $\nein\varphi_{\bot p}(\no a) = \nein 1_p = 0_p = \no 1_p = \no\varphi_{\bot p}(\nein a)$. It can be easily shown that for all $x\in A_\bot$ and $y\in A_p$, with $p\neq \bot$, we have that $x\le y \iff y = 1_p$ and $y\le x \iff y=0_p$.

Now, one can readily see that all the maps of $\widehat\Phi$ are monotone. As for~(lax), we have four cases: when $q\neq p=\bot \neq r$, and then $\nein\varphi_{\bot q}(a) = \nein 1_q = 0_q \le 1_r = \varphi_{\bot r}(\nein a)$;\footnote{~This is the only point in the proof where we use condition~(4).} the case $p=\bot=q\neq r$, in which $\nein\varphi_{\bot\bot}(a) = \nein a \le 1_r = \varphi_{\bot r}(\nein a)$; the case $p= \bot = r \neq q$, in which $\nein\varphi_{\bot q}(a) = \nein 1_q = 0_q \le \nein a = \varphi_{\bot\bot}(\nein a)$; and the case $p=q=r=\bot$, where $\nein\varphi_{\bot\bot}(a) = \nein a = \varphi_{\bot\bot}(\nein a)$.

We have just proved that $\m A$ is a substructure of $\m B$, which is also a locally integral ipo\nbd-semigroup. Moreover, since $\bot$ is the lower bound of $\m D$, then $1_\bot$ is the global identity of $\m B$. Hence, adding the constant $1_\bot$ to the signature of $\m B$ we obtain the locally integral ipo\nbd-monoid $\widehat{\m B}$, of which $\m A$ is a subreduct.

As for the final statement, we only need to choose $\m A_\bot$ to be an integral unital semiring. In order to check that $\m B$ contains all joins and meets, we only need to consider the case in which we have $a\in A_\bot$ and $b\in A_p$, with $p\in A^+$. We would have two different situations: if $b=0_p$ then $b\le 0_\bot\le a$, and therefore $a\lor b = a$. Otherwise, if $b\neq 0_p$ then $b\nle x$ for any $x\in A_\bot$, and $a\le y$, with $y\notin A_\bot$ only if $y\in A^+$. And since the least element of $A^+$ that is larger than $b$ is $1_p$, we have that $a\lor b = 1_p$. The case of the meets is analogous: $a\land 1_p = a$ and $a\land b = 0_p$, if $b\neq 1_p$.
\end{proof}

\begin{figure}
\centering\scriptsize
\begin{tikzpicture}[baseline=0pt]
\node at (0,-1.5)[n]{$\m 1$};
\node(0) at (0,0)[i,label=left:{$0=1$}]{};
\end{tikzpicture}
\ 
\begin{tikzpicture}[baseline=0pt]
\node at (0,-1.5)[n]{$\m 2$};
\node(1) at (0,1)[i,label=left:$1$]{};
\node(0) at (0,0)[i,label=left:$0$]{} edge (1);
\end{tikzpicture}
\ 
\noindent
\begin{tikzpicture}[baseline=0pt]
\node at (0,-1.5)[n]{$\m L_{3}$};
\node(2) at (0,2)[i,label=left:$1$]{};
\node(1) at (0,1)[label=right:$a$]{} edge (2);
\node(0) at (0,0)[i,label=right:$a^2$,label=left:$0$]{} edge (1);
\end{tikzpicture}
\ 
\begin{tikzpicture}[baseline=0pt]
\node at (0,-1.5)[n]{$\m 2^2$};
\node(3) at (0,2)[i,label=left:$1$]{};
\node(2) at (1,1)[i,label=left:$a$]{} edge (3);
\node(1) at (-1,1)[i,label=right:$-a$]{} edge (3);
\node(0) at (0,0)[i,label=left:$0$,label=right:$-a{\cdot}a$]{} edge (1) edge (2);
\end{tikzpicture}
\ 
\begin{tikzpicture}[baseline=0pt]
\node at (0,-1.5)[n]{$\m L_{4}$};
\node(3) at (0,3)[i,label=left:$1$]{};
\node(2) at (0,2)[label=right:$a$]{} edge (3);
\node(1) at (0,1)[label=right:$a^2$]{} edge (2);
\node(0) at (0,0)[i,label=right:$a^3$,label=left:$0$]{} edge (1);
\end{tikzpicture}
\ 
\begin{tikzpicture}[baseline=0pt]
\node at (0,-1.5)[n]{$\m I_{4,1}$};
\node(3) at (0,3)[i,label=left:$1$]{};
\node(2) at (0,2)[i,label=right:$a$]{} edge (3);
\node(1) at (0,1)[label=right:$b$]{} edge (2);
\node(0) at (0,0)[i,label=left:$0$,label=right:$b^2$]{} edge (1);
\end{tikzpicture}
\ 
\begin{tikzpicture}[baseline=0pt]
\node at (0,-1.5)[n]{$\m L_{5}$};
\node(4) at (0,4)[i,label=left:$1$]{};
\node(3) at (0,3)[label=right:$a$]{} edge (4);
\node(2) at (0,2)[label=right:$a^2$]{} edge (3);
\node(1) at (0,1)[label=right:$a^3$]{} edge (2);
\node(0) at (0,0)[i,label=left:$0$,label=right:$a^4$]{} edge (1);
\end{tikzpicture}
\ 
\begin{tikzpicture}[baseline=0pt]
\node at (0,-1.5)[n]{$\m I_{5,1}$};
\node(4) at (0,4)[i,label=left:$1$]{};
\node(3) at (0,3)[label=right:$a$]{} edge (4);
\node(2) at (0,2)[label=right:$b$]{} edge (3);
\node(1) at (0,1)[label=right:$a^2$][label=left:$ab$]{} edge (2);
\node(0) at (0,0)[i,label=left:$0$][label=right:$a^3$]{} edge (1);
\end{tikzpicture}
\ 
\begin{tikzpicture}[baseline=0pt]
\node at (0,-1.5)[n]{$\m I_{5,2}$};
\node(4) at (0,4)[i,label=left:$1$]{};
\node(3) at (0,3)[i,label=right:$a$][label=left:$ab$]{} edge (4);
\node(2) at (0,2)[label=right:$b$]{} edge (3);
\node(1) at (0,1)[label=right:$c$]{} edge (2);
\node(0) at (0,0)[i,label=left:$0$,label=right:$b^2$]{} edge (1);
\end{tikzpicture}
\ 
\begin{tikzpicture}[baseline=0pt]
\node at (0,-1.5)[n]{$\m P_{6,1}$};
\node(5) at (0,3)[i,label=left:$1$]{};
\node(4) at (.7,2)[i,label=right:$-{-}a$][label=left:$b$]{} edge (5);
\node(3) at (-.7,2)[i,label=left:$-{-}b$][label=right:$a$]{} edge (5);
\node(2) at (.7,1)[label=right:$-a$]{} edge (3) edge (4);
\node(1) at (-.7,1)[label=left:$-b$]{} edge (3) edge (4);
\node(0) at (0,0)[i,label=left:$0$]{} edge (1) edge (2);
\end{tikzpicture}
\caption{All integral involutive semirings up to size 5 and an integral ipo\nbd-monoid of size 6, as components for constructing locally integral idempotent semirings and ipo\nbd-semigroups.}
\label{fig:integral:ipomonoids:up:to:size:6}
\end{figure}
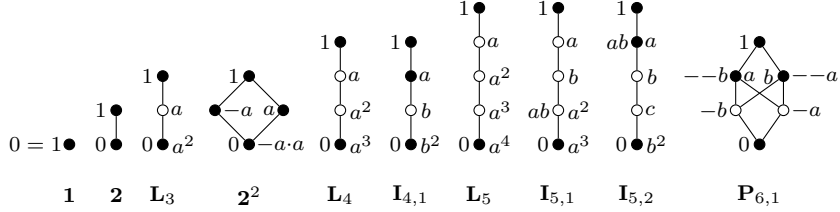

\section{Idempotent locally integral ipo-semigroups}

In this section we specialize to locally integral ipo\nbd-semigroups in which all elements are idempotent, i.e., the identity $x^2=x$ holds. The following two results describe this class of structures.

\begin{theorem}\label{thm:loc:integ:idempotent}
A locally integral ipo\nbd-semigroup is idempotent if and only if all its integral components are Boolean algebras.
\end{theorem}

\begin{proof}
A locally integral ipo\nbd-semigroup is idempotent if and only if all its integral components are idempotent. And an integral ipo\nbd-monoid is idempotent if and only if it is a Boolean algebra, because if $\m A$ is idempotent then for all $x,y\in A$, $x\cdot y = x\land y$. Indeed, $x\cdot y \le 1\cdot y = y$ and analogously $x\cdot y \le x$. And if $z\le x$ and $z\le y$, then $z = z\cdot z \le x\cdot y$.
\end{proof}

\begin{corollary}\label{cor:idempotent:commutative}
For any idempotent ipo\nbd-semigroup $\m A$, the following conditions are equivalent:
\begin{enumerate}
    \item $\m A$ satisfies $\no x\cdot x = x\cdot\nein x$,
    \item $\m A$ is commutative,
    \item $\m A$ has local identities,
    \item $\m A$ is cyclic,
    \item $\m A$ is locally integral and all its components are Boolean algebras.
\end{enumerate}
\end{corollary}

\begin{proof}
$(1 \implies 2)$ If $\m A$ satisfies $\no x\cdot x = x\cdot\nein x$, then it also satisfies $x\backslash x = x/x$, by Lemma~\ref{lem:char:local:bounds}. The inequality $yxx\le yx$ implies that $x\le yx\backslash yx = yx/yx$, and therefore $xyx\le yx$. Analogously, the inequality $yyx\le yx$ implies that $yxy\le yx$. Hence, $xy = x(yxy) \le xyx \le yx$, and thus $xy = yx$.

\vspace\itemsep
\noindent
$(2 \implies 3)$ If $\m A$ has local identities, in particular $x\backslash x = x/x$. Notice that the equation $(x/x)x=x$ holds in every idempotent ipo\nbd-semigroup, since $xx = x$ implies $x\le x/x$ and therefore $x = xx \le (x/x)x \le x$. Hence, $\m A$ has local identities.

\vspace\itemsep
\noindent
$(3 \implies 4)$ If $\m A$ has local identities, in particular it satisfies $\no x\cdot x = x\cdot\nein x$, by Lemma~\ref{lem:char:local:bounds}. By the previous argument, $\m A$ is commutative and therefore cyclic by Remark~\ref{rem:loc:ident:commutat:implies:cyclic}.

\vspace\itemsep
\noindent
$(4\implies 1)$ Since $x\le x/x = \no(x\cdot\no x)$, we get $x\cdot\no x\le \no x$ and hence $\no x\cdot x\le x$, by double negation. Analogously, since $x\le x\backslash x = \no(\no x\cdot x)$, we get $\no x\cdot x\le \no x$. Using these inequalities, we obtain $\no x \cdot x = \no x \cdot x\cdot \no x \cdot x \le x\cdot \no x$. That is, $\no x\cdot x\le x\cdot \no x$, and by double negation, $x\cdot\no x\le \no x\cdot x$, whence we derive the equality.

\vspace\itemsep
The fact that condition~(5) implies each one of the other conditions is trivial. Suppose then that $\m A$ satisfies condition~(1), and then it has local identities and it is commutative and cyclic. Notice also that $\m A$ satisfies conditions~(1), (3), and~(4) of Proposition~\ref{prop:char:locally:integral}. We will show next that all local identities are positive, what proves that $\m A$ is a commutative locally integral ipo\nbd-monoid, and by Theorem~\ref{thm:loc:integ:idempotent} all its components are Boolean algebras.

First, it is easy to check that, by residuation and idempotency, we have that $u/vw = (u/w)/v$ and $u/v \le uv/v$. Now, the inequality $0_x\le 1_x$ implies that $0_xyy = 0_xy \le 1_x y$, and by residuation, commutativity, contraposition, and the aforementioned properties, we have that
\begin{align*}
y &\le 1_xy/0_xy = (1_xy/y)/0_x = \no 0_x/\no (1_xy/y) = 1_x/(y\cdot\no(1_xy)) \\
&\le (1_xy\cdot\no(1_xy))/(y\cdot\no(1_xy)) = 0_{\no(1_xy)}/(y\cdot\no(1_xy)) = \big(0_{\no(1_xy)}/\no(1_xy)\big)/y \\
&= \no\no(1_xy)/y = (1_xy)/y,
\end{align*}
and by residuation and idempotency, $y = yy \le 1_xy$, as we wanted to prove.
\end{proof}

Since a commutative and idempotent semigroup is a semilattice, we will refer to these structures as locally integral \emph{ipo\nbd-semilattices}, or \emph{i$\ell$-semilattices} in case they are lattice-ordered. This semilattice order, called the \emph{multiplicative order}, is denoted by $\sqsubseteq$ and it is defined by $x\sqsubseteq y\iff xy=x$.

The smallest locally integral i$\ell$-semilattice that doesn't have a global unit is the $4$-element algebra described in Example~\ref{ex:loc:integ:no:global:unit}. And the smallest locally integral ipo\nbd-semilattice that is not lattice-ordered has $8$ elements. Their posets and multiplicative orders are shown in Figure~\ref{fig:cidil:cidipo}.

\begin{figure}[ht!]
\centering
\begin{tikzpicture}[scale=1, baseline=0pt]
\node(3) at (0,2)[label=above:$\top$]{};
\node(2) at (1,1)[label=right:$q$]{};
\node(1) at (-1,1)[label=left:$p$]{};
\node(0) at (0,0)[label=below:$\bot$]{};
\draw(0)--(1)--(3)--(2)--(0);
\node at (0,-1.5)[n]{$\le$};
\end{tikzpicture}
\qquad
\begin{tikzpicture}[scale=1, baseline=0pt]
\node(3) at (1,2)[label=right:$q$]{};
\node(2) at (-1,2)[label=left:$p$]{};
\node(1) at (0,1)[label=right:$\top$]{};
\node(0) at (0,0)[label=below:$\bot$]{};
\draw[ultra thick](0)--(1);
\draw(1)--(2)(1)--(3);
\node at (0,-1.5)[n]{$\sqsubseteq$};
\end{tikzpicture}
\qquad\qquad\qquad
\begin{tikzpicture}[scale=1, baseline=0pt]
\node(3) at (0,4)[label=above:$\top$]{};
\node(2) at (2,2){};
\node(1) at (-2,2){};
\node(0) at (0,0)[label=below:$\bot$]{};
\draw(0)--(1)--(3)--(2)--(0);
\node(7) at (-1,3){};
\node(6) at (1,1){};
\node(5) at (1,3){};
\node(4) at (-1,1){};
\draw(4)--(5)(6)--(7);
\node at (0,-1.5)[n]{$\le$};
\end{tikzpicture}
\qquad
\begin{tikzpicture}[scale=1, baseline=0pt]
\node(7) at (1,3){};
\node(6) at (2,2){};
\node(5) at (-1,3){};
\node(4) at (-2,2){};
\node(3) at (0,2)[label=above:$\top$]{};
\node(2) at (1,1){};
\node(1) at (-1,1){};
\node(0) at (0,0)[label=below:$\bot$]{};
\draw[ultra thick](0)--(1)--(3)--(2)--(0);
\draw[ultra thick](4)--(5)(6)--(7);
\draw(1)--(4)(3)--(5)(2)--(6)(3)--(7);
\node at (0,-1.5)[n]{$\sqsubseteq$};
\end{tikzpicture}
\caption{The smallest i$\ell$-semilattice that does not have a global identity and the smallest ipo\nbd-semilattice that is not lattice-ordered.}
\label{fig:cidil:cidipo}
\end{figure}
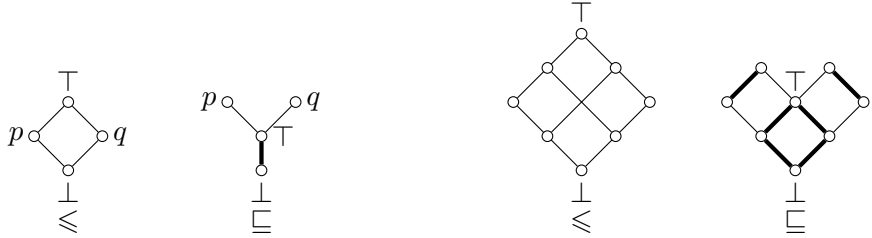

The number of locally integral ipo\nbd-semilattices with $n\le 16$ elements (up to isomorphism) have been calculated with Mace4~\cite{McC2010} and are given in Table~\ref{tab:iposemilattices:size-n}. Apart from the 16-element Boolean algebra, all these po-algebras can be constructed from glueings of Boolean algebras of cardinality $\le 8$.

\begin{table}
\centering\small\tabcolsep3pt
\begin{tabular}{l|cccccccccccccccc}
Number of elements $=$   &1&2&3&4&5& 6& 7& 8& 9&10&11&12&13&14&15&16\\\hline
ipo\nbd-semilattices      &1&1&1&3&4&10&17&43&82&201&&&&&&\\
i$\ell$-semilattices  &1&1&1&3&4&10&17&42&80&191&&&&&&\\
comm. idempotent ipo\nbd-monoids &1&1&1&2&2& 4& 4& 9&10&22&24&53&61&134&157&343\\
comm. idempotent i$\ell$-monoids &1&1&1&2&2& 4& 4& 9&10&21&22&49&52&114&121&270\\
Boolean algebras &1&1&0&1&0&0&0&1&0&0&0&0&0&0&0&1\\
\end{tabular}
\caption{Number of ipo\nbd-semilattices up to isomorphism. Boolean algebras are the building blocks for ipo\nbd-semilattices.}
\label{tab:iposemilattices:size-n}
\end{table}

To end this section, we will describe a duality for a certain class of locally integral ipo\nbd-semilattices including all finite ones. This dual representation gives a much more compact way of drawing finite ipo\nbd-semilattices (see Figures~\ref{fig:m14:4} and~\ref{fig:m14:8}). We first recall that generalized Boolean algebras form the variety \textbf{GBA} generated by the 2-element residuated lattice $\m 2_r$. An equational basis for this variety is given by the identities for residuated lattices together with $xy=x\land y$ and $(y\backslash x)/ y=x\lor y$ (see~\cite[Prop. 3.23]{GaJiKoHi07}). Actually, GBAs can be characterized as the residuated lattices for which every principal filter is a Boolean lattice (see~\cite{GaJiKoHi07}).

Finite members of \textbf{GBA} are reducts of Boolean algebras without complementation or a constant designating the bottom element. Homomorphisms between GBAs are residuated lattice homomorphisms, hence they preserve the top element $1$, join, meet, and residuals, but not necessarily complementation or $0$. A generalized Boolean algebra is \emph{complete} if all joins (and therefore all meets) exist, and it is \emph{atomic} if every non-bottom element is above some atom, or equivalently, if every element is the join of all the atoms below it.

Consider the category $\CAGBA$ whose objects are the complete and atomic generalized Boolean algebras and whose morphisms are the sup-preserving homomorphisms of generalized Boolean algebras. Consider also the category $\Par$ of sets and partial maps. The well-known Tarski duality between the category of complete and atomic Boolean algebras and the category of sets and maps can be extended to a duality between $\CAGBA$ and $\Par$ in the following way.

Consider the functor $\Psi\:\CAGBA\to\Par^\op$ defined as follows: for every $\m B$ in $\CAGBA$, $\Psi(\m B)=\At(\m B)$, the set of atoms of $\m B$, and for every morphism $\theta\:\m B\to \m B'$, $\Psi(\theta)\:\At(\m B')\parmap\At(\m B)$ is the partial map defined on $\{b\in\At(\m B') : b\nle\theta(0_\m B)\}$ as $\Psi(\theta)(b)$ being the only $a\in\At(\m B)$ such that $b\le\theta(a)$, and remaining undefined otherwise. Notice that this is a well-defined partial map, because if $b\in\At(\m B')$ and $a,a'\in\At(\m B)$ are such that $b\le \theta(a)$ and $b\le\theta(a')$, then $b\le \theta(a)\land \theta(a') =\theta(a\land a') =\theta(0_\m B)$. In the other direction, consider the functor $\P\:\Par^\op\to\CAGBA$ defined for every $X$ as $\P(X)$, the powerset generalized Boolean algebra, and for every partial map $f:X\parmap Y$, $\P(f)\:\P(Y)\to\P(X)$ is the \emph{partial inverse image} defined as $\P(f)(a) = U_f\cup\{x\in X : f(x)\in a\}$, where $U_f = \{x\in X : f(x)\text{ is undefined}\}$. These two functors establish the aforementioned categorical equivalence between $\CAGBA$ and $\Par^\op$.

Now, a complete and atomic generalized Boolean algebra can be understood as an integral ipo-monoid, and every sup-preserving homomorphism between complete and atomic generalized Boolean algebras is in particular a monoid homomorphism. Thus, we can view $\CAGBA$ as a subcategory of $\IIPOMn$. We would like to use the duality between $\CAGBA$ and $\Par$ and the equivalence of categories described in the previous section in order to give a duality between the locally integral ipo\nbd-semilattices whose components are complete and atomic and a certain category of semilattice directed systems of sets and partial maps.

The functor categories $\CAGBA^\SD$ and $(\Par^\op)^\SD$ are equivalent, since $\CAGBA$ and $\Par^\op$ are equivalent. That said, not every semilattice directed system of complete and atomic generalized Boolean algebras will give rise via the glueing construction to a locally integral ipo\nbd-semigroup, but only those satisfying~\eqref{eq:condition:on:zeros}, \eqref{eq:Phi:satisfies:bal}, (mon), and~(lax). Notice that, by assumption, all the considered homomorphisms are sup-preserving, and therefore monotone. Also, since these structures are commutative, condition~\eqref{eq:Phi:satisfies:bal} is trivially satisfied. Consider then the full subcategory $\SDCAGBA$ of $\CAGBA^\SD$ of those functors satisfying~\eqref{eq:condition:on:zeros} and~(lax). Using the equivalence of the previous section, this category $\SDCAGBA$ is equivalent to the category of locally integral ipo\nbd-semilattices which are locally complete and atomic. And by the duality just described, $\SDCAGBA$ is equivalent to a full subcategory $\SDPar$ of $(\Par^\op)^\SD$.

We note that conditions on semilattice directed systems of $\CAGBA$ 
translate into conditions on the semilattice directed systems of partial maps.
For instance, 
the condition~\eqref{eq:condition:on:zeros} corresponds to the following condition on a functor $F\:\ct D\to\Par^\op$ in $\SDPar$: for all $p\le q$ in $\ct D$,
\[
f_{pq}\: F(q)\parmap F(p) \text{ is total}\quad\iff\quad p=q.
\]

\begin{figure}[!ht]
\centering
\begin{tikzpicture}
\node(13) at (-1.5,1.5){};
\node(12) at (-1,1){};
\node(11) at (-3,0){};
\node(10) at (-2.5,-0.5){};
\node(9) at (-0.5, 0.5){};
\node(8) at (-2, -1){};
\draw[ultra thick] (13)--(11);
\draw[ultra thick] (8)--(10)--(12)--(9);
\draw(10)--(11)(12)--(13);
\draw[ultra thick] (8)--(9);
\node(7) at (0,0){};
\node(6) at (0,-1){};
\node(5) at (-1,-1){};
\node(4) at (-1,-2){};
\node(3) at (1,-1){};
\node(2) at (1,-2){};
\node(1) at (0,-2){};
\node(0) at (0,-3){};
\draw[ultra thick](4)--(5)--(7)--(6)--(4)
(0)--(1)--(3)--(2)--(0)
(0)--(4)(1)--(5)(2)--(6)(3)--(7);
\draw(9)--(7) (8)--(4);
\end{tikzpicture}
\qquad
\begin{tikzpicture}
\node at (0,5)[n]{};
\node(6) at (0,4){};
\node(5) at (.5,2.5){};
\node(4) at (-.5,2.5){};
\node(3) at (1,1){};
\node(2) at (-1,1){};
\node(1) at (0,1){};
\draw(4) edge (6);
\draw(2) edge (4);
\draw(1) edge (4);
\draw (0,2.5) ellipse [x radius=25pt, y radius=13pt];
\draw (0,1) ellipse [x radius=40pt, y radius=13pt];
\end{tikzpicture}
\qquad\qquad\qquad
\begin{tikzpicture}
\node(13) at (-0.5,1.5){};
\node(12) at (0,1){};
\node(11) at (-1,1){};
\node(10) at (-2.5,-0.5){};
\draw[ultra thick] (13)--(12);
\draw[ultra thick] (8)--(10)--(11)--(9);
\draw(9)--(12)(11)--(13);
\node(9) at (-0.5, 0.5){};
\node(8) at (-2, -1){};
\draw[ultra thick] (8)--(9);
\node(7) at (0,0){};
\node(6) at (0,-1){};
\node(5) at (-1,-1){};
\node(4) at (-1,-2){};
\node(3) at (1,-1){};
\node(2) at (1,-2){};
\node(1) at (0,-2){};
\node(0) at (0,-3){};
\draw[ultra thick](4)--(5)--(7)--(6)--(4)
(0)--(1)--(3)--(2)--(0)
(0)--(4)(1)--(5)(2)--(6)(3)--(7);
\draw(9)--(7) (8)--(4);
\end{tikzpicture}
\qquad
\begin{tikzpicture}
\node(6) at (0,4){};
\node(5) at (.5,2.5){};
\node(4) at (-.5,2.5){};
\node(3) at (1,1){};
\node(2) at (-1,1){};
\node(1) at (0,1){};
\draw(5) edge (6);
\draw(2) edge (4);
\draw(1) edge (4);
\draw (0,2.5) ellipse [x radius=25pt, y radius=13pt];
\draw (0,1) ellipse [x radius=40pt, y radius=13pt];
\end{tikzpicture}
\caption{Multiplicative orders and dual representations for two different ipo\nbd-monoids obtained from glueing the same Boolean algebras over a 3-element semilattice chain.}
\label{fig:m14:4}
\end{figure}
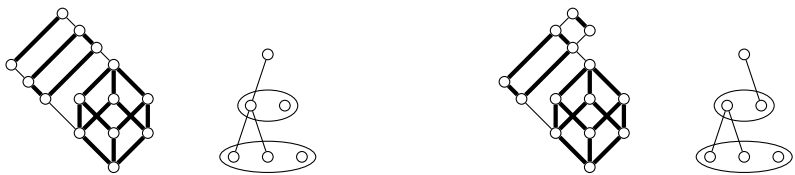

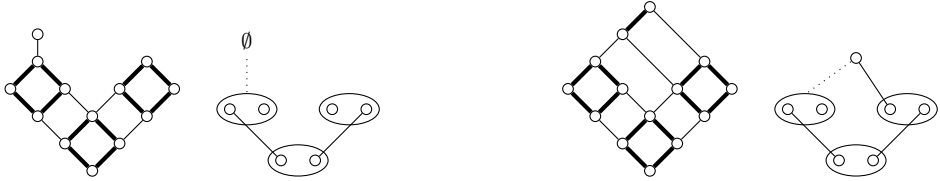
\begin{figure}[!ht]
\centering
\begin{tikzpicture}[scale=.8]
\node(12) at (-2,2){};
\node(11) at (2,1){};
\node(10) at (1, 0){};
\node(9) at (3, 0){};
\node(8) at (2,-1){};
\node(7) at (-2,1){};
\node(6) at (-3, 0){};
\node(5) at (-1, 0){};
\node(4) at (-2,-1){};
\node(3) at (0,-1){};
\node(2) at (-1, -2){};
\node(1) at (1, -2){};
\node(0) at (0,-3){};
\draw[ultra thick] (2)--(3)--(1)--(0)--(2)(6)--(7)--(5)--(4)--(6)(10)--(11)--(9)--(8)--(10);
\draw(1)--(8)(3)--(10)(2)--(4)(3)--(5)(12)--(7);
\end{tikzpicture}
\quad
\begin{tikzpicture}
\node(7) at (-1,4.5)[n]{\scriptsize$\emptyset$};
\node(6) at (-1.5,2.5){};
\node(5) at (-0.5,2.5){};
\node(4) at (2.5,2.5){};
\node(3) at (1.5,2.5){};
\node(2) at (1,1){};
\node(1) at (0,1){};
\draw[dotted](-1,3) edge (7);
\draw(2) edge (4);
\draw(1) edge (6);
\draw (0.5,1) ellipse [x radius=25pt, y radius=13pt];
\draw (2,2.5) ellipse [x radius=25pt, y radius=13pt];
\draw (-1,2.5) ellipse [x radius=25pt, y radius=13pt];
\end{tikzpicture}
\qquad\qquad\qquad
\begin{tikzpicture}[scale=.8]
\node(13) at (0, 3){};
\node(12) at (-1, 2){};
\node(11) at (2,1){};
\node(10) at (1, 0){};
\node(9) at (3, 0){};
\node(8) at (2,-1){};
\node(7) at (-2,1){};
\node(6) at (-3, 0){};
\node(5) at (-1, 0){};
\node(4) at (-2,-1){};
\node(3) at (0,-1){};
\node(2) at (-1, -2){};
\node(1) at (1, -2){};
\node(0) at (0,-3){};
\draw[ultra thick] (2)--(3)--(1)--(0)--(2)(6)--(7)--(5)--(4)--(6)(10)--(11)--(9)--(8)--(10)(12)--(13);
\draw(1)--(8)(3)--(10)(2)--(4)(3)--(5)(12)--(10)(12)--(7)(13)--(11);
\end{tikzpicture}
\quad
\begin{tikzpicture}
\node(7) at (.5,4){};
\node(6) at (-1.5,2.5){};
\node(5) at (-0.5,2.5){};
\node(4) at (2.5,2.5){};
\node(3) at (1.5,2.5){};
\node(2) at (1,1){};
\node(1) at (0,1){};
\draw[dotted](-.9,3) edge (7);
\draw(3) edge (7);
\draw(2) edge (4);
\draw(1) edge (6);
\draw (0.5,1) ellipse [x radius=25pt, y radius=13pt];
\draw (2,2.5) ellipse [x radius=25pt, y radius=13pt];
\draw (-1,2.5) ellipse [x radius=25pt, y radius=13pt];
\end{tikzpicture}
\caption{An example of an ipo\nbd-semilattice and an ipo\nbd-monoid, together with their dual representations.}
\label{fig:m14:8}
\end{figure}

Figures~\ref{fig:m14:4} and~\ref{fig:m14:8} contain several examples of locally integral ipo\nbd-semilattices. If $\m A$ is finite, then in particular $\m A^+$ and $\m A_p$ are finite, for all $p\in A^+$. Hence, its dual can be displayed as the Hasse diagram of $(\m A^+)^\op$ where every node $p$ of $\m A^+$ is replaced by the set $\At(\m A_p)$. If $q$ is a cover of $p$, the partial map $f_{pq}$ is given by edges from elements of $\At(\m A_q)$ to elements of $\At(\m A_p)$. The empty set is labeled by $\emptyset$ and the empty partial map is indicated by a dotted line between the sets. This graphical representation and the fact that systems of partial maps are logarithmically smaller than the corresponding finite ipo\nbd-semilattices make them a useful tool for investigating finite locally integral ipo\nbd-semilattices.

\nocite{*}
\bibliographystyle{fundam}
\bibliography{bibliography}

\end{document}